\documentclass[12pt,a4paper]{article}

\usepackage{amssymb,latexsym,amsmath,amsthm}
\usepackage{url}
\usepackage{movie15}

\usepackage{graphicx,hyperref}

\allowdisplaybreaks

\newcommand{\p}{\partial}

\renewcommand{\phi}{\varphi}
\newcommand{\e}{\epsilon}

\newcommand{\R}{{\mathbb R}}

\newcommand{\EX}{\operatorname{\mathbb E}}

\newcommand{\fu}{{\mathfrak u}}
\newcommand{\fB}{{\mathfrak B}}
\newcommand{\sign}{\operatorname{sign}}
\newcommand{\Var}{\operatorname{Var}}
\newcommand{\Z}[1]{\ensuremath{e^{#1}{\star}}}
\newcommand{\D}[2]{\mathchoice
  {\frac{\partial #2}{\partial #1}}
  {{\partial #2}/{\partial #1}}
  {{\partial #2}/{\partial #1}}
  {{\partial #2}/{\partial #1}}
  }
\newcommand{\Dn}[3]{\mathchoice
  {\frac{\partial^#2 #3}{\partial #1^#2}}%
  {{\partial^#2 #3}/{\partial #1^#2}}%
  {{\partial^#2 #3}/{\partial #1^#2}}%
  {{\partial^#2 #3}/{\partial #1^#2}}%
  }
\newcommand{\DD}[2]{\Dn{#1}2{#2}}

\newcommand{\Ord}[1]{\ensuremath{\mathcal O\big(#1\big)}}

\def\doi{}
\ifx\href\undefined
     \renewcommand{\doi}[1]{\url{http://dx.doi.org/#1}}
\else\renewcommand{\doi}[1]{\href{http://dx.doi.org/#1}{doi:#1}}
\fi

\newcommand{\cD}{{\cal D}}

\newcommand{\cL}{\ensuremath{\mathcal L}}

\newcommand{\pde}{\textsc{pde}}
\newcommand{\sde}{\textsc{sde}}
\newcommand{\spde}{\textsc{spde}}
\newcommand{\rat}[2]{{\textstyle\frac{#1}{#2}}}
\newcommand{\rxi}{\rat{\xi}{\sqrt2}}

\newtheorem{theorem}{Theorem}
\newtheorem{lemma}[theorem]{Lemma}
\newtheorem{coro}[theorem]{Corollary}
\theoremstyle{remark}
\newtheorem{remark}[theorem]{Remark}
\topsep=\parskip

\title{Self similarity and attraction in stochastic nonlinear reaction-diffusion systems}

\author{Wei Wang\thanks{School of Mathematics, University of Adelaide, South Australia, \textsc{Australia}.
\protect\url{mailto:w.wang@adelaide.edu.au}; 
and Department of Mathematics, Nanjing University, Nanjing, \textsc{China}.
\protect\url{mailto:wangweinju@yahoo.com.cn} }
\and
A.~J. Roberts\thanks{School of Mathematics, University of Adelaide, South Australia, \textsc{Australia}. 
\protect\url{mailto:anthony.roberts@adelaide.edu.au}}}

\date{6 November 2011}

\begin{document}

\maketitle

\begin{abstract}
Similarity solutions play an important role in many fields of science: we consider here similarity in stochastic dynamics.  Important issues are not only the existence of stochastic similarity, but also whether a similarity solution is dynamically attractive, and if it is, to what particular solution does the system evolve.  By recasting a class of stochastic PDEs in a form to which stochastic centre manifold theory may be applied we resolve these issues in this class.  For definiteness, a first example of self-similarity of the Burgers' equation driven by some stochastic forced is studied. Under suitable assumptions, a stationary solution is constructed which yields the existence of a stochastic self-similar solution for the stochastic Burgers' equation.  Furthermore, the asymptotic convergence to the self-similar solution is proved.  Second, in more general stochastic reaction-diffusion systems stochastic centre manifold theory provides a framework to construct the similarity solution, confirm its relevance, and determines the correct solution for any compact initial condition.  Third, we argue that dynamically moving the spatial origin and dynamically stretching time improves the description of the stochastic similarity.  Lastly, an application to an extremely simple model of turbulent mixing shows how anomalous fluctuations may arise in eddy diffusivities.  The techniques and results we discuss should be applicable to a wide range of stochastic similarity problems.
\end{abstract}

\tableofcontents


\section{Introduction}

Consider the stochastic field~$\fu(t,x)$ governed by the nonlinear reaction-diffusion stochastic partial differential equation (\spde) 
\begin{equation}
\fu_t=\fu_{xx}+ f(\partial_x,\fu,t^{-1/2})+ g(\partial_x,\fu,t^{-1/2})B(t,x)
\label{eq:prob}
\end{equation}
on an infinite spatial domain in one dimension.  Here the drift nonlinearity~$f(\partial_x,\fu,t^{-1/2})$ is cubic in the field~$\fu$, the spatial derivative, and decaying time~$t^{-1/2}$: stationary examples include the cubic reaction $f=-\fu^3$, the self-advection~$\fu\fu_x$ of Burgers' \pde, and the nonlinear diffusion~$\epsilon\sign(\fu_{xx})\fu_{xx}$ explored by Barenblatt~\cite[\S3.2.1]{Barenblatt96}.    For some stochastic process~$B$, with characteristics defined variously in later sections, the noise term has nonlinear coefficient~$g(\partial_x,\fu,t^{-1/2})$ which is quadratic in the field~$\fu$, the spatial derivative, and~$t^{-1/2}$: for example, the stationary linear advection~$\fu_x$.  Equation~\eqref{eq:prob} informally describes the scope of \spde{}s we consider.

Interpret all stochastic calculus in the Stratonovich sense so that, with care, the normal rules of calculus apply.

We apply some stochastic theory together with centre manifold theory to help understand and solve the long time evolution of such stochastic diffusion with nonlinear reaction-advection. But on the infinite spatial domain there is no clear cut slow eigenspace for \spde~\eqref{eq:prob}.  For example, in the absence of noise and nonlinearity, $f=g=0$\,, one substitutes spatial structures with wavenumber~$k$ into the diffusion \pde~\eqref{eq:prob}, $\fu\propto e^{ikx+\lambda t}$, to find the continuous spectrum $\lambda=-k^2$.  Such continuous spectra do not have a well defined decomposition between fast transients and the slow, long lasting, modes of interest.

A special transformation changes our view of the dynamics of \spde~\eqref{eq:prob} into one with a clear fast-slow separation.  We extend to stochastic dynamics the transformation used so interestingly by Wayne et al.~\cite{Beck2009, Wayne94, Wayne96} for deterministic similarity, and analogous deterministic analysis on appropriate initial conditions~\cite{Suslov98a}. Introducing log-time and similarity variables transforms the stochastic problem to one of seeking a stochastic field~$u(\tau,\xi)$ where
\begin{equation}
        \tau=\log {t}\,,\quad
        \xi=\frac{x}{\sqrt{t}}\,,\quad
        \fu=\frac{1}{\sqrt t}u(\tau,\xi),\quad
        \omega\in\Omega\,.
        \label{eq:gene}
\end{equation}
Then the dependence upon the scaled space variable~$\xi$ causes the diffusive Gaussian spread following a point release,
\begin{equation}
        \fu=\frac{a}{2\sqrt{\pi t}}\exp\left[\frac{-x^2}{4t}\right],
        \label{eq:gauss}
\end{equation}
to correspond to a fixed point of the dynamics for the stretched field~$u$, namely
\begin{equation}
        u_*=\frac{a}{2\sqrt{\pi}}\exp\left[\frac{-\xi^2}{4}\right].
        \label{eq:fix}
\end{equation}
Further, we shall see that the algebraic decay in real time~$t$ from any compact release to the  Gaussian~\eqref{eq:gauss} transforms to an exponentially quick decay in log-time~$\tau$ to the fixed point~\eqref{eq:fix}.  Solutions of the \spde~\eqref{eq:prob} that are approximately the self-similar Gaussian spread will be approximated based upon analysis about the fixed points~\eqref{eq:fix} in log-time

A similar, but stochastic, version of the transformation~\eqref{eq:gene} shows the appearance of an eddy diffusion model of the long time behaviour of a simple turbulent mixing model introduced by Majda, McLaughlin, Camassa et al.~\cite{Majda93, McLaughlin96, Bronski2000, Camassa2008}.  Whereas they primarily explored solutions statistically stationary in space, as is the theme in this article, Section~\ref{s:saxft} characterises the spread following a compact release.  Although a mean eddy diffusion emerges, significant anomalous diffusion occurs throughout due to stochastic fluctuations.

\begin{figure}
\centering
\includemovie[controls,poster,text={View a movie in Acrobat Reader.}]{\linewidth}{\linewidth}{similarity3.mpg}
\\[2ex]
\caption{movie of stochastic self-similarity emerging in a realisation of a stochastic Burgers' equation.  Here simply obtained via a naive but fine scale discretisation of the \spde~\eqref{e:txBurgers} for noise $\fB\propto \dot W(t,x)/t$.  The top frame show the decay and spread in physical variables, whereas the bottom frame shows the same realisation in similarity variables~\eqref{eq:gene} and hence the approach to a stochastic self-similarity.}
\label{fig:similarity3}
\end{figure}

Section~\ref{sec:sbe} proves the existence and emergence of stochastic self-similarity for a stochastic Burgers' equation in the class of the \spde~\eqref{eq:prob}.
Burgers' deterministic partial differential equation for a field~$\fu (t,x)$ is the case $f=-\fu\fu_x$ and $g=0$ of the \spde~\eqref{eq:prob}, namely
$\fu _t= \fu _{xx}-\fu \fu _x$\,.
This \pde\ was proposed by Burgers'~\cite{Burgers} to illustrate the statistical theory of turbulent fluid motion. To better model the turbulence and turbulent flow in the presence of random forces,  stochastic Burgers' equations have been suggested~\cite{CAMSS88,CTMK92,  Hay00,  HY75, Jeng69,  Sinai91, Sinai92} and  studied by many people~\cite{Bertini94, BD07, Truman, E99, Holden94, Prato94, Prato95,   Truman08}.  We explore the stochastic solutions~$\fu (t,x,\omega)$ to the stochastic Burgers' \pde
\begin{equation}\label{e:txBurgers}
\fu _t= \fu _{xx}-\fu \fu _x+\fB(t,x)
\end{equation}
where $\fB(t,x)$ is some stochastic force,  to be detailed later, defined on a complete probability space~$(\Omega, \mathcal{F}, \mathbb{P})$.
Figure~\ref{fig:similarity3} shows just one realisation of solutions to Burgers' \spde~\eqref{e:txBurgers} to illustrate the emergence of stochastic self-similarity.
On \emph{bounded} domains Da Prato et al.~\cite{Prato94} proved the existence and uniqueness of global solution  when the noise term~$\fB(t,x)$ is a noise white in time and fixed spatial structure. Holden et al.~\cite{Holden94} also derived the same result by some white noise calculus. On an \emph{unbounded} domain Bertini et al.~\cite{Bertini94} constructed a global solution by a Cole--Hopf transformation with space-time white noise~$\fB(t,x)$.

Here we consider a family of solutions with special spatio-temporal form, \emph{stochastically self-similar solutions}, of the stochastic Burgers' equation~(\ref{e:txBurgers}) on the unbounded real line with a particular form for the stochastic force~$\fB(t,x)$. The existence of self-similar solutions and the asymptotic emergence of a self-similar solutions, Figure~\ref{fig:similarity3}, describes the self-similarity of the stochastic Burgers' equation.

Self-similarity is an important property of some convective diffusion equations, of which Burgers' equation is a special case. Many researchers have established the existence of self-similar solutions of deterministic systems~\cite{Xin96, Zuazua91, Liu85, Suslov98a,  Zuazua94}, and described the asymptotic behaviour of such self-similar solutions~\cite{Beck2009,  Zuazua91, Miller,  Zuazua94}.  Here for stochastic partial differential equation, we construct a self-similar solution in the sense of distribution, which is called a random  self-similar solution. Previously, Eyink and Xin~\cite{Xin99} studied the deterministic similarity of statistical quantities.  In contrast, here we additionally explore the structures of the stochastic fluctuations.

\subsection{A stochastic slow manifold emerges}
\label{sec:ssme}

Let's look at the transformation~\eqref{eq:gene} applied to the general reaction-diffusion \spde~\eqref{eq:prob}.  First, consider the coefficient functions: because the reaction term~$f$ is assumed cubic in its arguments $f(\partial_x,\fu,t^{-1/2})=t^{-3/2}f(\partial_\xi,u,1)$; similarly, as $g$~is assumed quadratic in its arguments the noise coefficient $g(\partial_x,\fu,t^{-1/2})=t^{-1}g(\partial_\xi,u,1)$.  Second, the original \spde~\eqref{eq:prob} thus transforms to the \spde
\begin{equation}
        u_\tau=\cL u+f(\partial_\xi,u,1)
        +g(\partial_\xi,u,1)\dot W,
        \label{eq:phit}
\end{equation}
where the new noise $\dot W\equiv\sqrt t B$ is assumed to be $Q$-Wiener space-time noise, and where the linear operator
\begin{equation}
        \cL u=u_{\xi\xi}+\rat{1}{2}\xi u_\xi+\rat{1}{2}u\,.
        \label{eq:lphi}
\end{equation}
The long time dynamics of the original \spde~\eqref{eq:prob} are revealed by the dynamics of the transformed \spde~\eqref{eq:phit}.

Consider the \spde~\eqref{eq:phit} for negligible noise, $\dot W= 0$\,. The \spde\ then has equilibrium $u=0$\,.  Linearised about this equilibrium the \spde\ has discrete spectrum $\lambda_k=-k/2$\,, $k=0,1,2,\ldots$\,, with the corresponding eigenfunctions
\begin{equation}
e_k(\xi)=c_kH_k(\xi/\sqrt2)G(\xi),
\quad\text{for Gaussian } 
G(\xi)=\frac1{2\sqrt\pi}\exp\left[\frac{-\xi^2}4\right],
\label{eq:eigen}
\end{equation}
where the Hermite polynomials $H_k(\zeta)=(-1)^ke^{\zeta^2/2}\Dn \zeta k{e^{-\zeta^2/2}}$.
These eigenfunctions forms a standard orthonormal basis: we choose the normalisation constant~$c_k=(2\sqrt\pi/k!)^{1/2}$ as then $\int_{-\infty}^\infty e_k(\xi)^2e^{\xi^2/4}\,d\xi=1$\,.
The wonderful aspect of this transformation to log-time is that the spectrum of the linear operator~$\cL$ is discrete and there exists a clear separation of the fast modes, mode numbers $k\geq1$\,, from the slow mode, $k=0$\,.
That is, the equilibrium $u=0$ of the \spde~\eqref{eq:phit} has the slow subspace $u=aG(\xi)$ for all~$a$.
We use the existence of this slow subspace as the basis for analysing emergent stochastic self-similarity. 

We suppose the noise is $Q$-Wiener with the following spectral decomposition
\begin{equation}
\dot W(\tau,\xi)=\sum_{k=0}^\infty b_k\dot w_k(\tau)e_k(\xi)
=\sum_{k=0}^\infty b_k\dot w_k(\tau)H_k(\xi/\sqrt2)G(\xi),
\label{eq:dotw}
\end{equation}
where $w_k(\tau)$ are independent Wiener processes, the noise coefficients~$b_k$ decay sufficiently rapidly with mode number~$k$ (for previous rigorous theory only a finite number of noise coefficients~$b_k$ can be non-zero), and the overall magnitude of the noise terms are denoted by some norm $b=\|\vec b_k\|$.
Because of the spectrum, the conditions on the noise, and for smooth enough coefficients $f$~and~$g$, there must exist a stochastic slow manifold~\cite{Boxler89, Arnold03}.

Furthermore, under suitable conditions on the noise and the \spde\ coefficients $f$~and~$g$, the stochastic slow manifold is exponentially quickly attractive to all nearby initial conditions.  Thus stochastic self-similarity of solutions to the \spde~\eqref{eq:prob} emerges from generic compact initial conditions.




\section{A stochastic Burgers' equation}
\label{sec:sbe}

This first extensive section proves in detail the emergence and nature of stochastic self-similarity in the stochastic Burgers' equation~\eqref{e:txBurgers}.  Subsequent sections on other \spde{}s are less rigorous but apply more widely. 

We here construct and analyse a stochastic self-similar solution for the stochastic Burgers' equation~(\ref{e:txBurgers}) for $t\geq 1$ with a special stochastic force~$ \eta$.  For this we invoke the self-similarity transform~\eqref{eq:gene},
then
\begin{equation}\label{e:tau-xi-Burgers}
du=\left[ u_{\xi\xi}+\rat12\xi u_\xi+\rat12 u-uu_\xi\right]d\tau+dW(\tau,\xi)
\end{equation}
with initial $u(0,\xi)=u^{0}(\xi)=\fu(1,\xi)$.
Stationary solutions~$ \bar{u}(\xi,\omega)$ to equation~(\ref{e:tau-xi-Burgers}) is a self-similar solution of stochastic Burgers' equation~(\ref{e:txBurgers}).
Here in order to construct a self-similar solution of stochastic Burgers' equation~(\ref{e:txBurgers}) we assume that stochastic process~$ W(\tau, \xi)$ is an~$L^2(\R)$ valued Q-Wiener process defined on~$ (\Omega,\mathcal{F}, \mathbb{P})$ with covariance operator~$Q$ which is detailed later.

To construct a stationary solution of~(\ref{e:tau-xi-Burgers}), we consider the system in a weighted space~$ L^2(K)$ defined in the next subsection~\ref{ss:p}. First by using energy estimates and the compact embedding results of the weighted space, we show the tightness of solution with initial value in the space~$ L^2(K)$.  Then the classical Bogolyubov--Krylov method~\cite{Arnold03} implies the existence of stationary solution of~(\ref{e:tau-xi-Burgers}).  Further, we show the locally attractive property of
the self-similar solution by a Cole--Hopf transformation (Theorem \ref{thm:u-ubar}) and a local random invariant manifold method~(Theorem~\ref{cor:u-ubar}). The Cole--Hopf transformation makes the equation~(\ref{e:tau-xi-Burgers}) to be a linear one which shows that the stationary solution is determined uniquely by the mass of the solution provide the stationary solution is small enough in the space~$ L^2(\Omega, L^\infty(\R))$. Then every small solution in the space~$ L^2(\Omega, L^\infty(\R))$ is attracted by a unique stationary solution. This is also shown by a local random invariant manifold discussion without assuming solution is small in~$L^\infty(\R)$.  The last subsection~\ref{ss:gacsss} shows  the globally asymptotic convergence in probability which shows the existence of global random invariant manifold of equation~(\ref{e:tau-xi-Burgers}).


\subsection{Preliminary}
\label{ss:p}

Consider the stochastic \textsc{pde}~(\ref{e:tau-xi-Burgers}).  For this we use the linear operator~$\cL$ defined in equation~\eqref{eq:lphi}.
Define the exponentially increasing weight function $K(\xi)=\exp\{\xi^2/4\}$, and then introduce the following weighted  functional spaces for $p>0$
\begin{equation*}
L^p(K)=\left\{u\in L^p(\R): \|u\|^p_{L^p(K)}=\int_\R|u(\xi)|^pK(\xi)\,d\xi<\infty\right\}
\end{equation*}
and for positive integer~$ k$
\begin{equation*}
H^k(K)=\left\{u\in L^k(K): \|u\|^2_{H^k(K)}=\sum_{0\leq\alpha\leq k}\|D^\alpha u\|^2_{L^2(K)}<\infty\right\}\,.
\end{equation*}
Then linear operator~$\cL $ is self-adjoint and  generates an analytic semigroup~$S(t)$ on the space~$L^2(K)$ with the domain $D(\cL )=H^2(K)$~\cite{Kavian}. Recall that the eigenvalues of operator~$ \cL$ are $\lambda_k=-{k}/{2}$\,, $k=0,1,2,\ldots$\,,
with the corresponding eigenfunctions~\eqref{eq:eigen} forming a standard orthonormal basis of~$L^2(K)$ as $\|e_k\|_{L^2(K)}=1$\,.

Define inner product on space $L^{2}(K)$ as 
\begin{equation*}
\langle u, v\rangle=\int_{\R}u(\xi)v(\xi)K(\xi)\,d\xi\,,\quad u\,, v\in L^{2}(K).
\end{equation*}
Then denote by $P_{0}$ and $P_{s}$ the linear projection from $L^{2}(K)$ to the slow subspace $E_{0}$ and the stable subspace $E_{s}$ respectively. Then 
$E_0=\text{span}\{e_0(\xi)\}$ and 
\begin{equation*}
E_s=E_{0}^{\bot}=\left\{u\in L^2(K): \int_\R u(\xi)\,d\xi=0\right\}\,.
\end{equation*}
This first lemma presents some basic properties on these weighted spaces~\cite{Kavian}.
\begin{lemma}\label{lem:Kavian}
\begin{enumerate}
  \item The embedding $H^1(K)\subset L^2(K)$ is compact.
  \item There exists $C>0$ such that for any $u\in H^1(K)$
      \begin{equation*}
      \int_\R|u(\xi)|^2|\xi|^2K(\xi)\,d\xi\leq C\int_\R|\nabla u(\xi)|^2K(\xi)\,d\xi\,.
      \end{equation*}
  \item \label{i:3} For any $u\in H^1(K)$,
   \begin{equation*}
    \rat12\int_\R|u(\xi)|^2K(\xi)\,d\xi\leq \int_\R|\nabla u(\xi)|^2K(\xi)\,d\xi\,.
   \end{equation*}
  \item For any $u\in E_s$\,,
   \begin{equation*}
    \left\langle \cL  u, u \right\rangle\leq -\rat12 \|u\|^2_{H^1(K)}\,.
   \end{equation*}
  \item If $u\in H^1(K)$, then $K^{1/2}u\in L^\infty(\R)$.
  \item \label{i:6}  For any $q>2$\,, $\e>0$ there exists constants $C_{\e,q}>0$\,, $R>0$\,, such that for any $u\in H^{1}(K)\cap L^{q}_{\text{loc}}$
 \begin{equation*}
 \|u\|^{2}_{L^{2}(K)}\leq \e \|u_{\xi}\|^{2}_{L^{2}(K)}+C_{\e,q} \|u\|^{2}_{L^{q}(B(0, R))}\,.
 \end{equation*} 
\end{enumerate}
\end{lemma}
\begin{remark}\label{rem:norm}
By part~\ref{i:3} in the above lemma, in the space~$ H^1(K)$ we can define norm~$\|\nabla u\|_{L^2(K)}$  which is equivalent to~$\|u\|_{H^1(K)}$\,.
\end{remark}
Further, by the spectrum property of the linear operator~$\cL $  we define~$ (\cL-1/2) ^\gamma$ for any $\gamma\in\R$~\cite{Yosida}. Then define Sobolev space~$H^{\gamma}(K)$, for any $\gamma\in\R$ as~$\cD ((\cL-1/2) ^{\gamma/2})$, the domain of~$ (\cL-1/2) ^{\gamma/2}$, and by the embedding theorem~\cite{Yosida}, $H^{\gamma_1}$~is compactly embedding into~$H^{\gamma_2}$ for  $\gamma_1>\gamma_2$\,.

For our purpose we assume that under the self-similar variable the special stochastic force~$\eta$ is assumed to be an $L^2(K)$-valued Wiener process~$W(\tau, \xi)$, $\tau\geq 0$\,, which is conservative; that is, it satisfies
\begin{equation}\label{e:intW}
\int_\R W(\tau ,\xi)\,d\xi=0\,.
\end{equation}
Moreover, $W(\tau, \xi)$ has the following spectral expansion
\begin{equation}\label{e:Wiener}
W(\tau, \xi)=\sum_{k=1}^\infty b_kw_k
(\tau)e_k(\xi)
\end{equation}
where $\{w_k(\tau)\}_k$ are mutually independent standard scalar Wiener processes. By the conservation assumption~(\ref{e:intW}), $b_0=0$ and so is omitted from this sum.  So  $W(\tau,\xi)=W_s(\tau,\xi)$ the part of~$W(\tau,\xi)$ in the stable  subspace~$E_s$.
Furthermore we assume
\begin{equation}\label{e:traceQ}
\sum_{k=1}^\infty b^{2}_k<
\infty\,.
\end{equation}
For later purpose we extend $W(\tau,\xi)$ to the whole time interval $(-\infty, +\infty)$ by $W(-\tau,\xi)=-W(\tau, \xi)$, $\tau\geq 0$\,.  In the following we consider equation~(\ref{e:tau-xi-Burgers}) on  the canonical  probability space $(\Omega_0, \mathcal{F}_0, \mathbb{P}_{0})$ which consists of the sample path of~$ W(\cdot,\omega)$ in the space~$ C(\R, L^2(K))$, and denote by~$ \{\theta_\tau\}_\tau$ the Wiener shift on~$ \Omega_0$~\cite{Arnold03}. 

\begin{remark}
The special assumption on~$ B(t,x)$ does not exclude the existence of self-similar solution for  other cases. 
\end{remark}

\begin{remark}
Notice here that $\Omega_{0}$ is different from $\Omega$. However by the self similar transformation,   for any $\omega\in\Omega$\,, there is a sample path of $W(\tau,\xi)$, then there is a $\omega_{0}\in\Omega_{0}$ represents one sample path of $W(\tau,\xi)$ in $\Omega_{0}$\,.   The inverse is same. So in the following we see $\Omega_{0}$ same as $\Omega$\,.
\end{remark}

To estimate the solution to~(\ref{e:tau-xi-Burgers}) we need the
Ornstein--Uhlenbeck process~$\eta^\alpha(\tau,\xi)$ which solves the following equation
\begin{equation}\label{OU}
d\eta^\alpha=[\cL \eta^\alpha-
\alpha\eta^\alpha]\,d\tau+dW(\tau,\xi),
\end{equation}
for some $\alpha>0$\,;
that is,
$\eta^\alpha(\tau,\xi)=\int_{-\infty}^\tau e^{(\cL -\alpha)(\tau-s)}dW(s)$.
(Here we introduce the differential~$dW(s)$ to denote the log-time differential~$dW(s,\xi)$ but without explicitly including space~$\xi$ to denote the differential does not involve~$\xi$.)
Then we have the following estimates on~$ \eta^\alpha(\tau,\xi)$.
\begin{lemma}\label{lem:OU-H1}
For any $\e>0$ and $p\geq 2$ there exist $\alpha>0$ such that
\begin{equation*}
\EX\left(\|\eta^\alpha(\tau)\|^p_{H^1(K)}  \right)<\e\,.
\end{equation*}
\end{lemma}
\begin{proof}
We just need to prove this result for $p=2j$\,, $j\geq 1$\,.

First the unique stationary solution to equation~(\ref{OU}) is written as
\begin{equation*}
\eta^\alpha(\tau)=\sum_{k=1}^\infty
\sqrt{b_k}\int_{-\infty}^\tau e^{(-\lambda_k-\alpha)(\tau-s)}dw_k(s)e_k\,.
\end{equation*}
Then by the property of stochastic integral,
\begin{equation*}
\EX\left(\|\eta^\alpha_\xi(\tau)\|^2_{L^2(K)}  \right)=\sum_{k=1}^\infty \frac{b_k\lambda_k}{\lambda_k+\alpha}\,.
\end{equation*}
By the assumption~(\ref{e:traceQ}) on~$W(\tau, \xi)$,  Lemma~\ref{lem:Kavian} and Remark~\ref{rem:norm} we have the result for $j=1$ by choosing~$\alpha$ large enough. For $j>1$\,, the result is followed by the Guassian property of $\eta^\alpha$~\cite[Lemma 7.2]{Prato} .
\end{proof}
\begin{remark}
Here the $\alpha$ is chosen to be  large so that $\eta^\alpha$ is small.  Subsection~\ref{subsec:diffusion} shows,  for $\alpha=0$\,, that the process~$\eta^{\alpha}$ is the emergent stochastic slow manifold of linear system~(\ref{OU}) with $\alpha=0$ and describes the stochastic self similarity of the linear problem~(\ref{eq:prob}) with $f=0$ and $g=1/t$\,.
\end{remark}

Let $\eta^\alpha=\eta^\alpha_0+\eta^\alpha_s\in E_0\oplus E_s$\,, then by the assumption~(\ref{e:intW}), $\eta_0^\alpha=0$\,. By Lemma~\ref{lem:Kavian} and~\ref{lem:OU-H1}, we then have
\begin{coro}\label{coro:OU-L-infty}
For any $\e>0$, there exists $\alpha>0$ such that for any $p\geq 1$
\begin{equation*}
\EX\|\eta^\alpha(\tau)\|^p_{L^\infty(\R)}<\e\,.
\end{equation*}

\end{coro}


\subsection{Existence of a stochastic  self-similar solution}\label{sec:selfsimilar}
For any $\tau>0$\,, in the mild sense equation~(\ref{e:tau-xi-Burgers}) is written as
\begin{equation}\label{e:mild}
u(\tau)=S(\tau)u^{0}+\int_0^\tau S(\tau-s)u(s)u_\xi(s)\,ds+\int_0^\tau S(\tau-s)\,dW(s).
\end{equation}
Then by the theory for abstract stochastic evolutionary equations~\cite{Prato} we have the following theorem.
\begin{theorem}\label{thm:wellpose}
For any $T>0$ and $u^0\in L^{2}(\Omega_{0}, L^2(K)\cap L^\infty(\R))$, there is a unique solution~$u(\tau, \xi)$ to \spde~(\ref{e:tau-xi-Burgers}) in the space 
$L^2(\Omega_0,C(0, T; L^2(K))\cap L^2(0, T; H^1(K)))$.
\end{theorem}
\begin{proof}
Notice that the nonlinearity $F(u)=uu_\xi$ is locally Lipschitz continuous from $L^2(K)\rightarrow H^{-\gamma}$ for some $0<\gamma<1$\,. Then by a cutoff technique we have the existence and uniqueness of local solution for some stopping time~$T$ by the classical application of Banach Fixed Point Theorem for stochastic evolutionary equation~\cite{Prato}.
The global existence and uniqueness is thus followed by the a priori estimates in following part. 
\end{proof}
 
To construct a stationary solution by Bogolyubov--Krylov method, we need the tightness of~$\{\cL (u(\tau,\xi))\}_{\tau>0}$ in the space~$L^2(K)$. By the compact embedding of~$H^1(K)\subset L^2(K)$, we just give the uniform estimates of~$\{u(\tau,\xi)\}_{\tau>0}$ in the space~$H^1(K)$. For this we need some estimate in the spaces~$L^\infty(\R)$ and~$L^2(K)$.

\subsubsection{Estimates in the space $L^\infty(\R)$}\label{sec:L-infty}
 
   We adapt 
the approach used by Zuazua~\cite{Zuazua94}  for a scalar convection-diffusion equations to give an $L^\infty(\R)$~estimate for solution to equation~(\ref{e:tau-xi-Burgers}) with initial value $u^0\in L^2(K)\cap L^\infty(\R)$.

For this we introduce
\begin{equation*}
\text{sgn}(u)^+=\begin{cases}
1,& u>0\,, \\ 0,&  u\leq 0\,;
\end{cases}
\quad\text{and}\quad
\text{sgn}(u)^-=\begin{cases}
1,& u<0\,, \\  0,&  u\geq 0\,.
\end{cases}
\end{equation*}
Then for $u\in L^2(\R)$ with $\nabla u(t)\in L^2(\R)$,
$\int_\R u_{\xi\xi}\phi(u)\,d\xi<0$
for any nondecreasing $\phi\in C^1(\R)$. By a density discussion  
$\int_\R u_{\xi\xi} \text{sgn}(u)^\pm\,d\xi\leq 0$\,.
Moreover, $\int_\R uu_\xi\text{sgn}(u)^\pm\,d\xi=0$
and $\int_\R (\xi u_\xi+u)\text{sgn}(u)^\pm\,d\xi=0$\,. Denote by $u^{\pm}=\text{sgn}(u)^{\pm} u$\,. 
Now let $m(\tau)=\|u^{0}\|_{L^{\infty}(\R)}+\|\eta(\tau)\|_{L^{\infty}(\R)}$ with $\eta(\tau)$ solves~(\ref{OU}) with $\alpha=0$\,. Then multiplying  $\text{sgn}(u-m)^+$ on both sides of~(\ref{e:tau-xi-Burgers}), and integrating on $\R\times [0, \tau]$ with $\tau>0$  
\begin{equation*}
\int_\R(u(\tau,\xi)-m)^+\,d\xi\leq 0\,.
\end{equation*}
Therefore $u(\tau,\xi)\leq m(\tau)$ for any $\tau>0$. Similarly 
$u(\tau,\xi)\geq -m(\tau)$\,, $\tau>0$.
Then  
$\|u(\tau)\|_{L^\infty(\R)}\leq m(\tau)$\,, for all $\tau>0$.
Moreover, by the property of $\eta^{\alpha}(\tau)$\,,  $\EX m(\tau)$ is bounded by  some positive constant.

\subsubsection{Estimates in the space  $H^1(K)$}
We first give estimates of solution $\{u(\tau,\xi)\}_{\tau>0}$ in the space~$L^2(K)$. 
 
First for~$u(\tau, \xi)$ solving~(\ref{e:tau-xi-Burgers}) with $u^{0}\in L^\infty(\R)\cap L^2(K)$ we decompose $u(\tau,\xi)=u_0(\tau, \xi)+u_s(\tau,\xi)$ with $u_0\in E_0$ and $u_s\in E_s$\,. Then by the assumption~(\ref{e:intW})
\begin{eqnarray*}
du_0&=&0\,,\\
du_s&=&\left[\cL  u_s-P_{s}(uu_\xi)\right]dt+dW_2\,.
\end{eqnarray*}
Then  
\begin{equation*}
u_0(\tau,\xi)=u_0(0, \xi)=\langle u(0), e_0\rangle e_0(\xi)=\int_\R u(0, \xi)\,d\xi e_0(\xi).
\end{equation*}
So $u_0(\tau, \xi)$ is totally determined by the mass of initial value which is denoted by $M=\int_\R u(0, \xi)\,d\xi$\,.

Now introduce $v=u-\eta^\alpha$\,, then we have the following random evolutionary equation
\begin{equation}\label{e:v}
v_\tau=\cL v-(v+\eta^\alpha)
(v+\eta^\alpha)_\xi+\alpha\eta^\alpha. \end{equation}
We first  give a uniform estimate on~$v(\tau)$ in the space~$L^2(K)$. Similarly write $v=v_0+v_s\in E_0\oplus E_s$\,, then
 \begin{equation*}
 v_0(\tau,\xi)=u_0(\tau,\xi)=Me_0(\xi).
 \end{equation*}
 Multiplying~$v$ in the space~$L^2(K)$ on both sides of the equation~(\ref{e:v}),  
\begin{equation*}
\rat12\frac{d}{d\tau}\|v_s\|^2_{L^2(K)}=
    -\rat12\|v_s\|^2_{H^1(K)}+\alpha\langle \eta_s^\alpha, v_s\rangle-\rat12\int_\R(v+\eta^\alpha)^2
\frac{\p}{\p\xi}(vK)\,d\xi\,.
\end{equation*}
Consider the third term on left hand side of the above equality   
\begin{eqnarray*}
\rat12\int_\R(v+\eta^\alpha)^2\frac{\p}{\p\xi}(vK)\,d\xi
&=&-\int_\R v^2v_\xi K\,d\xi-\int_\R v\eta^\alpha v_\xi K\,d\xi\\&&{}-\int_\R v^2\eta^\alpha_\xi K\,d\xi-\int_\R\eta^\alpha\eta^\alpha_\xi vK\,d\xi\,.
\end{eqnarray*}
We estimate the left four terms separately. By Cauchy inequality for any $\e>0$\,, there is positive constant~$C_\e$ such that
\begin{eqnarray*}
\left|\int_\R v\eta^\alpha v_\xi K\,d\xi\right|&\leq& \|\eta^\alpha\|_{L^\infty}\|v\|_{L^2(K)}
\|v_\xi\|_{L^2(K)}\\
&\leq&\e\|v_\xi\|^2_{L^2(K)}+C_\e\|\eta^\alpha\|^2_{L^\infty(\R)}
\|v\|^2_{L^2(K)}\,,\\
\left|\int_\R v^2\eta^\alpha_\xi K\,d\xi\right|&\leq&
\|v\|_{L^\infty}\|v\|_{L^2(K)}\|\eta^\alpha_\xi\|_{L^2(K)}
\\&\leq& \|\eta^\alpha\|_{H^1(K)} \|v\|^2_{L^2(K)}\,,\\
\left|\int_\R\eta^\alpha\eta^\alpha_\xi vK\,d\xi\right|&\leq& \|\eta^\alpha\|_{L^\infty(\R)}\|\eta^\alpha_\xi\|_{L^2(K)}\|v\|_{L^2(K)}\\
&\leq&C_\e\|\eta^\alpha\|^4_{H^1(K)}+\e\|v\|^2_{L^2(K)}\,.
\end{eqnarray*}
Integrating by parts  
\begin{equation*}
\int_\R v^2v_\xi K\,d\xi=\int_\R v^2K\,dv=-2\int_\R v^2v_\xi K\,d\xi-\int_\R v^3 K_\xi\,d\xi\,.
\end{equation*}
By property \ref{i:6} in  Lemma \ref{lem:Kavian}, for any $\e$\,, $\e'>0$, $q>2$\,, there exist positive constants $C_{\e}$\,, $C_{\e', q}$ and $R$ such that 
 \begin{eqnarray*}
\left|\int_{\R}v^{3}K_{\xi}\,d\xi\right|&=&\frac12\left|\int_{\R}v\xi K^{1/2} v^{2}K^{1/2}\,d\xi\right|\\
&\leq & \frac12\left[\int_{\R}v^{2}\xi^{2}K\,d\xi\right]^{1/2}\left[\int_{\R}v^{4}K\,d\xi\right]^{1/2}\\
&\leq& C\|v_{\xi}\|_{L^{2}(K)}\|v\|_{L^{2}(K)}\|v\|_{L^{\infty}(\R)}\\
&\leq&3\e C\|v_{\xi}\|^{2}_{L^{2}(K)}+3C_{\e}\left[\e'\|v_{\xi}\|^{2}_{L^{2}(K)}+C_{\e',q}\|v\|^{2}_{L^{q}(B(0, R))}\right]\|v\|^{2}_{L^{\infty}(\R)}\\&\leq &
3\left[\e C+\e' C_{\e}\|v\|^{2}_{L^{\infty}(\R)}\right]\|v_{\xi}\|^{2}_{L^{2}(K)}+3C_{\e}C_{\e', q, R}\|v\|^{4}_{L^{\infty}(\R)}
\end{eqnarray*}
with some positive constant $C_{\e', q, R}$\,,
then
\begin{eqnarray*}
\left|\int_\R v^2 v_\xi K\,d\xi\right|&=&\left|-\frac13\int_\R v^3 K_\xi\,d\xi\right|\\&\leq& \left[ \e C+ \e'C_{\e}\|v\|^{2}_{L^{\infty}(\R)}\right]\|v_{\xi}\|^{2}_{L^{2}(K)}+C_{\e}C_{\e', q, R}\|v\|^{4}_{L^{\infty}(\R)}\,.
\end{eqnarray*} 

%
Then for any~$\e$ and $\e'>0$\,, there is a positive constant that we still denote by~$C_\e$ such that
\begin{eqnarray*}
&&\rat12\frac{d}{d\tau}\|v_s(\tau)\|^2_{L^2(K)}\\&\leq&
-\rat12\|v_s(\tau)\|^2_{H^1(K)}+\e\|v_s(\tau)\|^2_{L^2(K)}
+\alpha C_\e\|\eta_s^\alpha(\tau)\|^2_{L^2(K)}+
\e\|v\|^2_{H^1(K)}\\&&{}+ 
\left[\e C+ \e'C_{\e}\|v\|^{2}_{L^{\infty}(\R)}\right]\|v\|^{2}_{H^{1}(K)}+C_{\e}C_{\e', q, R}\|v\|^{4}_{L^{\infty}}
 \\&&{}+C_\e\|\eta^\alpha\|^2_{L^\infty(\R)}
\|v\|^2_{L^2(K)}+\|\eta^\alpha\|_{H^1(K)}\|v\|^2_{L^2(K)}
+\e\|v\|^2_{L^2(K)}+C_\e\|\eta^\alpha\|^4_{H^1(K)}\\
&\leq&\left[-\rat12+\e(1+C)+\e'C_{\e}\|v\|^{2}_{L^{\infty}(\R)}\right]\|v_s\|^2_{H^1(K)}
+\left[2\e +C_\e\|\eta^\alpha\|^2_{L^\infty}\right.\\&&{}\left.+\|\eta^\alpha\|_{H^1(K)}
\right]\|v_s\|^2_{L^2(K)}+\left[-\rat12+\e(1+C)+\e' C_{\e}\|v\|^{2}_{L^{\infty}(\R)}\right]\|v_0\|^2_{H^1(K)}\\&&{}
+\left[2\e +C_\e\|\eta^\alpha\|^2_{L^\infty(\R)}+\|\eta^\alpha\|_{H^1(K)}
\right]\|v_0\|^2_{L^2(K)}
\\&&{}+C_\e\|\eta^\alpha\|^4_{H^1(K)}
+\alpha C_\e\|\eta^\alpha\|^2_{L^2(K)}+C_\e C_{\e', q, R}\|v\|^4_{L^\infty(\R)}\,.
\end{eqnarray*}
Now choosing  $\e$ and $\e'>0$ small enough and noticing that
$\|v\|_{L^2(K)}\leq 2\|v\|_{H^1(K)}$\,,  
\begin{eqnarray*}
\frac{d}{d\tau}
\|v_s(\tau)\|^2_{L^2(K)}\leq
\left[-\rat12 +H(\tau,\omega)  \right]\|v_s(\tau)\|^2_{L^2(K)}
+h(\tau,\omega)
\end{eqnarray*}
where
\begin{equation*}
H(\tau,\omega)=2\left[\tfrac{\e}{2}(3+C)+\e'C_{\e}\|v\|^{2}_{L^{\infty}(\R)}+C_\e\|\eta^\alpha\|^2_{L^\infty(\R)}
+\|\eta^\alpha\|_{H^1(K)}\right]
\end{equation*}
and
\begin{eqnarray*}
h(\tau,\omega)&=&\left[-\rat12+H(\tau,\omega)  \right]M^2+\alpha C_\e\|\eta^\alpha(\tau)\|^2_{L^2(K)}
\\&&{}+C_\e\|\eta^\alpha(\tau)\|^4_{H^1(K)}+C_\e C_{\e', q, R}\|v\|^4_{L^\infty(\R)}\,.
\end{eqnarray*}
Then by the Gronwall inequality for any $\tau>0$
\begin{eqnarray}\label{e:est-v2}
\|v_s(\tau)\|^2_{L^2(K)}&\leq&e^{-\rat12\tau+
\int_0^\tau H(s)\,ds}
\|v(0)\|^2_{L^2(K)}\\
&&{}+\int_0^\tau e^{-\tfrac12(\tau-s)+
\int_s^\tau H(\varsigma)\,d\varsigma}h(s)\,ds\,.\nonumber
\end{eqnarray}
By the construction of~$\eta^\alpha(\tau,\omega)$ and  the estimate on $u(\tau)$ in $L^{\infty}(\R)$~(section \ref{sec:L-infty})\,, 
choose $\e$ and $\e'>0$ small enough there is a random variable~$\tau_0(\omega)$ such that for $\tau>\tau_0$ almost surely
\begin{equation*}
\exp\left\{-\tfrac12\tau+\int_0^\tau H(s)\,ds\right\}\leq \exp\left\{-\tfrac14\tau\right\}\,.
\end{equation*}
Then  $ v(\tau)$~is time uniformly bounded by a tempered random variable~$R_1(\theta_\tau\omega)$ in the space~$L^2(K)$.

\begin{remark}
Notice that~$\eta^\alpha$  is  very small in~$L^2(\Omega_0, H^1(K))$ by choosing large $\alpha>0$ (Lemma~\ref{lem:OU-H1}). Then by the definition of~$h$ and estimates of~$u(\tau, \xi)$ in the space~$L^\infty(\R)$,  for small initial value $u_0\in L^2(K)\cap L^\infty(\R)$,  $h$~is small  in~$L^2(\Omega_0)$.  Then by estimate~(\ref{e:est-v2}) for $\tau>\tau_0$ is large, $v(\tau)$~is small in the space~$L^2(K)$ almost surely.
\end{remark}

Now given any $\tau_1>0$\,, in the mild sense  
\begin{eqnarray*}
v(\tau+\tau_1)&=&S(\tau)v(\tau_1)+\tfrac12\int_0^\tau S(\tau-\sigma)v(\tau_1+\sigma)\,d\sigma\\&&{}
+\int_0^\tau S(\tau-\sigma)(v(\tau_1+\sigma)+
\eta^\alpha(\tau_1+\sigma))\\&&\quad{}\times(v(\tau_1+\sigma)
+\eta^\alpha(\tau_1+\sigma))_\xi\,d\sigma\\&&{}
+\alpha\int_0^\tau S(\tau-\sigma)\eta^\alpha(\sigma)\,d\sigma\,.
\end{eqnarray*}
Then taking $H^1(K)$~norm in the above equation, by the growth of the semigroup, for some positive constant~$C$
\begin{eqnarray*}
&&\|v(\tau+\tau_1)\|_{H^1(K)}\\&\leq& C(1+\frac{1}{\sqrt{\tau}})\|v(\tau_1)\|_{L^2(K)}
+\int_0^\tau(1+\frac{1}{\sqrt{\tau-\sigma}})e^{-\tau/2}
R_1(\theta_\sigma\omega)\,d\sigma
\\&&{}+C\int_0^\tau(1+\frac{1}{\sqrt{\tau-\sigma}})
e^{-\tau/2}R_1(\theta_\sigma\omega)
\|v(\tau_1+\sigma)\|_{H^1(K)}\,d\sigma\\
&&{}+C\int_0^\tau(1+\frac{1}{\sqrt{\tau-\sigma}})
e^{-\tau/2}(1+R_1(\theta_\sigma\omega))\eta^\alpha(\theta_\sigma\omega)\,d\sigma\,.
\end{eqnarray*}
By Gronwall lemma, and the tempered property of~$R_1(\theta_\tau\omega)$, there is  a  tempered random variable~$R_2(\theta_\tau\omega)$ such that
\begin{equation*}
\|v(\tau+\tau_1)\|_{H^1(K)}\leq CR_2(\theta_\tau\omega)(1+\frac{1}{\sqrt{\tau}})\|v(\tau_1)\|_{L^2(K)} +R_2(\theta_\tau\omega)
\end{equation*}
for any $\tau_1>0$\,. Then we have the uniform estimate of~$v(\tau)$ in the space~$H^1(K)$ and the compact embedding of $H^1(K)\subset L^2(K)$ with the property of~$\eta^\alpha$ yields  the tightness of~$\{\cL (u(\tau))\}_{\tau\geq 0}$, the laws of~$u(\tau)$,  in the space~$L^2(K)$.  Then the classical Bogolyubov--Krylov method~\cite{Arnold03} yields the existence of an invariant measure denoted by~$\mu$. 

 Now choose random variable~$u^{0}$ with $\cL (u^{0})=\mu$\,, then the solution $\bar{u}(\xi,\omega)$ with initial value $u^{0}$ is a stationary solution to~(\ref{e:tau-xi-Burgers}).  By the self similar transformation for $t\geq 1$\,,
\begin{equation*}
\bar{\fu }(t, x,\omega)=\frac{1}{\sqrt{t}}\bar{u}\left(x/\sqrt{t}, \omega\right)=\frac{1}{\sqrt{t}}\bar{u}\left(x/\sqrt{t},\omega\right)
\end{equation*}
is a random self-similar solution to stochastic Burgers' equation~(\ref{e:txBurgers}).
\begin{remark}
By the stationary property of~$\bar{u}$,  $\sqrt{t}\bar{\fu }(t,x,\omega)$  just depends on $\xi=x/\sqrt{t}$ in the sense of distribution. So $\bar{\fu }(t,x,\omega)$ is called the random  self-similar solution to stochastic Burgers' equation~(\ref{e:txBurgers}).
\end{remark}

\subsection{Locally asymptotic convergence to self-similar solutions}
Self-similar solutions are important in study the dynamics of the system~(\ref{e:txBurgers}). Next we show that  solution~$u(\tau, \xi)$ to equation~(\ref{e:tau-xi-Burgers})  tends to  a unique stationary solution~$\bar{u}(\tau, \xi)$ as $\tau\rightarrow\infty$ under some conditions. This shows the asymptotical convergence of stochastic Burgers' equation~(\ref{e:txBurgers}) to the self-similar solution.

First we have the following result on the stationary solution to equation~(\ref{e:tau-xi-Burgers}).
\begin{lemma}\label{lem:unique-stationary}
Any stationary solution, which is small in~$L^2(\Omega_0, L^\infty(\R)\cap L^2(K))$, to~(\ref{e:tau-xi-Burgers}) is uniquely determined by it's mass,  that is for any given $M\in\R$\,, there is a unique stationary solution~$\bar{u}(\tau, \xi)$ to~(\ref{e:tau-xi-Burgers}) with
\begin{equation}\label{e:M}
\int_\R\bar{u}(\tau,\xi)\,d\xi=M\,.
\end{equation}
\end{lemma}
\begin{proof}
Suppose $\bar{u}_1$ and~$\bar{u}_2$ are two stationary solutions to~(\ref{e:tau-xi-Burgers}) with
\begin{equation}\label{e:u1=u2}
\int_\R\bar{u}_1(\tau,\xi)\,d\xi=
\int_\R\bar{u}_2(\tau,\xi)\,d\xi=M\,.
\end{equation}
Let $U=\bar{u}_1-\bar{u}_2$\,, then
\begin{equation*}
U_\tau=\cL U-(\bar{u}_1U)_\xi-UU_\xi\,.
\end{equation*}
By the Cole--Hopf transformation
\begin{equation}\label{e:ColeHopf}
V=U\exp\left\{-\tfrac12\int_{-\infty}^\xi U(y)\,dy\right\}
\end{equation}
we have
\begin{equation}\label{e:VV}
V_\tau=\cL V-(\bar{u}_1V)_\xi\,.
\end{equation}
Notice~(\ref{e:u1=u2}),  $V$ has zero projection to~$E_0$, that is $V\in E_s$\,. Multiplying~$V$ on both sides of~(\ref{e:VV}) in~$L^2(K)$  
\begin{equation*}
\rat12\frac{d}{dt}\|V\|^2_{L^2(K)}=-\rat12\|V\|^2_{H^1(K)}
-\int_\R\bar{u}_1 VV_\xi K\,d\xi-\rat12
\int_\R\bar{u}_1 V^2\xi K\,d\xi\,.
\end{equation*}
For the last two terms by Cauchy inequality for any $\e>0$\,, for some positive constant~$C_\e$
\begin{equation*}
\left|\int_\R\bar{u}_1 VV_\xi K\,d\xi\right|\leq
C_\e\|\bar{u}_1\|^2_{L^\infty(\R)}\|V\|^2_{L^2(K)}+\e \|V_\xi\|^2_{L^2(K)}
\end{equation*}
and
\begin{eqnarray*}
\left|\int_\R\bar{u}_1 V^2\xi K\,d\xi\right|
&\leq &\|\bar{u}_1\|_{L^\infty(\R)}\int_\R V^2\xi K\,d\xi\\
&\leq&\|\bar{u}_1\|_{L^\infty(\R)}\left[\int_\R V^2K\,d\xi\right]^{\rat12}\left[
\int_\R V^2\xi^2K\,d\xi\right]^{\rat12}\\
&\leq& C_\e\|\bar{u}_1\|^2_{L^\infty(\R)}\|V\|^2_{L^2(K)}
+\e\|V_\xi\|^2_{L^2(K)}\,.
\end{eqnarray*}
Then by $\|V\|_{L^2(K)}\leq 2\|V\|_{H^1(K)}$,  for small~$\e$
\begin{equation*}
\frac{d}{dt}\|V\|^2_{L^2(K)}
=\left(-\rat12+\rat32\e+3C_\e\|\bar{u}_1\|^2_{L^\infty(\R)}
\right)\|V\|^2_{L^2(K)}\,.
\end{equation*}
Now for small~$\bar{u}_1$ in~$L^2(\Omega_0, L^\infty(\R)\cap L^2(K))$, by the stationary property and the Gronwall inequality,   almost surely
\begin{equation*}
V(\tau)\rightarrow 0\,,\quad \tau\rightarrow \infty
\end{equation*}
which yields the uniqueness of the stationary solution satisfying~(\ref{e:M}). The proof is complete.
\end{proof}

By the above result we show that the long time behavior of some solution of~(\ref{e:tau-xi-Burgers}) is approximated by a unique stationary solution. For this we first show that the stationary solution~$\bar{u}$ constructed by Bogolyubov--Krylov method is bounded by  the bound of~$u(\tau,\xi)$.

\begin{lemma}\label{lem:baru-u}
For any solution~$u(\tau, \xi)$ to equation~(\ref{e:tau-xi-Burgers}) with initial value $u^{0}\in L^2(K)\cap L^\infty(\R)$, and  
 $\EX\|u(\tau,\xi)\|^2_{L^2(K)\cap L^\infty(\R)}\leq C$ for  some $\tau$-independent  positive constant~$C$, then there is a stationary solution~$\bar{u}$ such that 
\begin{equation*}
\int_\R\bar{u}(\xi)\,d\xi=\int_\R u(\tau,\xi)\,d\xi
\quad\text{and}\quad
\EX\|\bar{u}\|^2_{L^2(K)\cap L^\infty(\R)}\leq C\,.
\end{equation*}

\end{lemma}
\begin{proof}
Denote by $\mu_s=\cL u(s)$ the distribution of~$u(s,\cdot)$ in the space~$L^2(K)$. The Bogolyubov--Krylov method introduces the following probability measure on~$L^2(K)$,
\begin{equation*}
\bar{\mu}_\tau=\frac{1}{\tau}\int_0^\tau \mu_s\,ds \,,
\end{equation*}
and finds a limit point of~$\{\bar{\mu}_\tau\}_{\tau>0}$\,.
By the estimates of solution in the space~$H^1(K)$, $\{\bar{\mu}_\tau\}_{\tau>0}$ is tight in the space~$L^2(K)$, then there is a probability measure~$\bar{\mu}$ on the space~$L^2(K)$ and  subsequence~$\tau_n$ with $\tau_n\rightarrow\infty$\,, $n\rightarrow\infty$\,,  such that 
\begin{equation*}
\bar{\mu}_{\tau_n}\rightarrow \bar{\mu}\,,\quad n\rightarrow\infty
\end{equation*}
in weak sense~\cite{Bill} .
Now let~$\bar{u}$ be the stationary solution to equation~(\ref{e:tau-xi-Burgers}) with initial distribution~$\bar{\mu}$, then for for any $\tau>0$\,,
\begin{equation*}
\EX\|\bar{u}\|^2_{L^2(K)}=\int_{L^2(K)}\|\tilde{u}(\tau,\xi; v)\|^2_{L^2(K)}\,\bar{\mu}(dv)
\end{equation*}
wherever $\tilde{u}(\cdot,\xi; v)$ is the solution to equation~(\ref{e:tau-xi-Burgers}) with initial value $v\in L^2(K)$. 
By the construction of~$\bar{\mu}$   
\begin{eqnarray*}
\EX\|\bar{u}\|^2_{L^2(K)}&=&\lim_{n\rightarrow\infty} \frac{1}{\tau_n}\int_0^{\tau_n} \int_{L^2(K)}\|\tilde{u}(\tau,\xi; v)\|^2_{L^2(K)}\,\mu_s(dv)\,ds\\
&\leq&\lim_{n\rightarrow\infty} \frac{1}{\tau_n}\int_0^{\tau_n} \EX\|\tilde{u}(\tau,\xi;u(s,\xi))\|^2_{L^2(K)}\,ds\\
&\leq &C\,.
\end{eqnarray*}
Similar  $\EX\|\bar{u}\|^2_{L^\infty(\R)}\leq C$\,.
The proof is complete.
\end{proof}
Now we prove the following local asymptotical convergence to self-similar solution.  

\begin{theorem}\label{thm:u-ubar}
For any solution~$u(\tau, \xi)$, which is small in~$L^2(\Omega_0, L^\infty(\R)\cap L^2(K))$ for any $\tau>0$\,, there is a unique stationary solution~$\bar{u}(\tau, \xi)$ such that almost surely 
\begin{equation*}
\|u(\tau)-\bar{u}(\tau)\|_{L^{2}(K)}\rightarrow 0\,,\quad \tau\rightarrow\infty\,.
\end{equation*}
\end{theorem}
\begin{proof}
Assume for any $\tau>0$\,,  $\EX\|u(\tau,\xi)\|^2_{L^2(K)}\leq C$  for some positive constant~$C$.
Let
\begin{equation*}
M=\int_\R u(\tau, \xi)\,d\xi\,,
\end{equation*}
then by Lemma~\ref{lem:unique-stationary} and Lemma~\ref{lem:baru-u}\,, there is a unique stationary solution, denoted by~$\bar{u}(\tau, \xi)$,  satisfies
\begin{equation*}
\int_\R\bar{u}(\tau,\xi)\,d\xi=M\,.
\end{equation*}
Moreover,  $\bar{u}$ is small in the space  $L^2(\Omega_0, L^2(K)\cap L^\infty(\R))$. 

Let $U(\tau,\xi)=u(\tau,\xi)-\bar{u}(\tau,\xi)$, then
\begin{equation*}
U_\tau=\cL U-(\bar{u}U)_\xi-UU_\xi\,.
\end{equation*}
By the Cole--Hopf transformation~(\ref{e:ColeHopf}), we similarly have
\begin{equation}\label{e:VV1}
V_\tau=\cL V-(\bar{u}V)_\xi\,.
\end{equation}
By the choice of~$\bar{u}$, $V\in E_s$\,, and by same discussion in the proof of Lemma~\ref{lem:unique-stationary},  
\begin{equation*}
\frac{d}{dt}\|V\|^2_{L^2(K)}
=\left(-\rat12+\rat{3}{2}\e+3C_\e\|\bar{u}\|^2_{L^\infty(\R)}
\right)\|V\|^2_{L^2(K)}\,.
\end{equation*}
By the stationary property of~$\bar{u}(\tau,\xi)$  
almost surely
 \begin{equation*}
\|V(\tau,\xi)\|_{L^2(K)}\rightarrow 0\,,\quad \tau\rightarrow\infty\,.
\end{equation*}
The proof is complete.

\end{proof}
Then for stochastic Burgers' equation~(\ref{e:txBurgers}) on time interval $t\geq 1$  
\begin{theorem}
For $t\geq 1$\,, if the solution~$\fu (t,x)$ to equation~(\ref{e:txBurgers}) satisfies that $e^{\tau/2}\fu (e^{\tau}, e^{\tau/2}\xi)$ is small in the space  $L^2(\Omega_0, L^{2}(K)\cap L^{\infty}(\R))$ for any $\tau>0$\,, then there is a unique stochastic self-similar solution~$\bar{\fu }$ such that for almost all $\omega\in\Omega$
\begin{equation*}
\sqrt{t}\|\fu (t,x,\omega)-\bar{\fu }(t,x,\omega)\|_{L^{2}(\R)}\rightarrow 0\,,\quad t\rightarrow\infty
\end{equation*}
where 
\begin{equation*}
\bar{\fu }(t,x,\omega)=\frac{1}{\sqrt{t}}\bar{u}\left(x/\sqrt{t},\omega\right)
\end{equation*}
with $\bar{u}$ is the unique stationary solution to equation~(\ref{e:tau-xi-Burgers}) with $\int_{\R}\bar{u}(\xi)\,d\xi=\int_{\R}\fu (1,\xi)\,d\xi$\,.
\end{theorem}

\begin{remark}
By the above discussion, if $u$~is smaller in $L^2(\Omega_0, L^\infty(\R)\cap L^2(K))$, $\e$~can be chosen smaller and then the exponential converge rate is closer to~$-1/2$\,.
\end{remark}


\subsection{Asymptotic convergence  described by a local random invariant manifold }\label{sec:rds}
Next we show the locally asymptotical convergence  of the solution~$u(\tau,\xi)$  by the random invariant manifold theory~\cite[e.g.]{Arnold03, DuanLuSchm}.  This approach does not  require the solution~$u(\tau,\xi)$ is small in the space~$L^2(K)\cap L^\infty(\R)$.

Notice that the nonlinearity~$uu_\xi$ is local Lipschtiz, a global random invariant manifold is difficult to be constructed~\cite{DuanLuSchm, DuanLuSchm2}. Recent work by Bl\"omker and Wang~\cite{BlomkerWang10} gave a cut-off method to construct a local random invariant manifold~(\textsc{lrim}) for \textsc{spde}s with quadratic nonlinearity. Here we also consider a \textsc{lrim} for equation~(\ref{e:tau-xi-Burgers}) by a cut-off technique.

We construct a \textsc{lrim} for a  stationary solution~$\bar{u}$\,.  
Notice the noise is additive in~(\ref{e:tau-xi-Burgers}), for any stationary solution~$\bar{u}$ we introduce
\begin{equation*}
v(\tau,\xi)=u(\tau,\xi)-\bar{u}(\theta_\tau\omega, \xi).
\end{equation*}
Then denote by $B(u,v)=\tfrac12(uv_\xi+vu_\xi)$,  
\begin{equation}\label{e:RBurgers}
v_\tau=\cL v-[B(v+\bar{u}, v+\bar{u})-B(\bar{u}, \bar{u})]
\end{equation}
which is a random evolutionary equation.
By the property of~$\cL$, deterministic approach for the well-posedness yields that equation~(\ref{e:RBurgers}) defines a continuous random dynamical system~$\phi(\tau,\omega)$ with driven system $\{\theta_{\tau}\}_{\tau}$.

Now let~$P_s$ be the projection from~$L^2(K)$ to the space~$E_s$, and
\begin{equation*}
v=v_0+v_s\in E_0\oplus E_s\,.
\end{equation*}
Then
\begin{eqnarray*}
\dot{v}_0&=&0\\
\dot{v}_s&=&\cL v_s-[B(v+\bar{u}, v+\bar{u})-B(\bar{u},\bar{u})]\,.
\end{eqnarray*}
Notice that nonlinearity
\begin{equation*}
F(v,\theta_\tau\omega)=B(v+\bar{u}, v+\bar{u}(t))-B(\bar{u},\bar{u})
\end{equation*}
is non-Lipschtiz, to construct a \textsc{lrim} for~$\phi(t,\omega)$ we need some cutoff technique~\cite{BlomkerWang10} .  Denote by~$B_R(0)$ the ball with radius~$R$ that is
\begin{equation*}
B_R(0)=\{u\in L^2(K): \|v\|_{L^2(K)}\leq R\}\,.
\end{equation*}
For any set $A\subset L^2(K)$, define the following  distance
\begin{equation*}
\operatorname{dist}(u, A)=\inf_{v\in A}\|u-v\|_{L^2(K)}
\end{equation*}
for any $u\in L^2(K)$.
Then introduce the following cutoff function  on the space ~$L^2(K)$
\begin{equation*}
B^R(u,u)=\chi_R(u)B(u,u),\quad u\in L^2(K)
\end{equation*}
where $R>0$ and $\chi_R(u)=\chi(u/R)$ with $\chi: L^2(K)\rightarrow \R$ is a smooth bounded function with $\chi(u)=1$ if $\|u\|_{L^2(K)}\leq 1$ and $\chi(u)=0$ if $\|u\|_{L^2(K)}\geq 2$\,. Then the nonlinear term $B^R(u, u): L^2(K)\rightarrow H^{-\gamma}(K)$, $0<\gamma<1$\,, is Lipschitz continuous. More precisely for some positive constant~$C_B$
\begin{equation*}
\|B(u_1, u_1)-B(u_2, u_2)\|_{H^{-\gamma}(K)}\leq 2RC_B\|u_1-u_2\|_{L^2(K)}
\end{equation*}
for all $u_1,u_2\in L^2(K)$ with $\|u_1\|_{L^2(K)}\leq R$ and $\|u_2\|_{L^2(K)}\leq R$\,. Now define
\begin{equation*}
F^R(v,\theta_\tau\omega)=B^R(v+\bar{u},v+\bar{u})-
B^R(\bar{u}, \bar{u}),
\end{equation*}
then
\begin{equation*}
\text{Lip}_{L^2(K), H^{-\gamma}(K)}(F^R(u,u))=L_R:=2RC_B\,.
\end{equation*}

Consider now the following cut-off system
\begin{eqnarray}\label{e:v-c-cut}
\dot{v}^R_0&=&0\\
\dot{v}^R_s&=&\cL v^R_s-F^R(v,\theta_\tau\omega)\label{e:v-s-cut}
\end{eqnarray}
which defines a continuous random dynamical system~$\phi^R(\tau, \omega)$.
By the spectrum property of~$\cL$, for small $R>0$\,, the  Lyapunov--Perron method for \textsc{spde}s is applicable to the cutoff system~(\ref{e:v-c-cut})--(\ref{e:v-s-cut})~\cite[e.g.]{BlomkerWang10, DuanLuSchm2}\,. Then for enough small~$R$, $\phi^R(\tau,\omega)$~has a random invariant manifold~$\mathcal{M}^R_{\text{cut}}(\omega)$ which can be represented by
\begin{equation*}
\mathcal{M}_{\text{cut}}^R(\omega)=\{(v_0, h(v_0,\omega)): v_0\in E_0\}\,.
\end{equation*}
Here $h(\cdot, \omega): E_0\rightarrow E_s$ is Lipschitz continuous with $h(0,\omega)=0$ and
\begin{eqnarray*}
h(v_0,\omega)=-\int_{-\infty}^0
e^{-\cL s}F^R(\bar{v}_0^*+\bar{v}^*_s,\theta_s\omega)\,ds
\end{eqnarray*}
where $(\bar{v}^*_0, \bar{v}_s^*)$ is the unique solution of
\begin{eqnarray}\label{e:v*c}
v^*_0(\tau)&=&v_{0}\\
v^*_s(\tau)&=&-\int_{-\infty}^\tau e^{\cL(\tau-s)}F^R(v_0^*+v^*_s,\theta_s\omega)\,ds\label{e:v*s}
\end{eqnarray}
in the Banach space
\begin{equation*}
C_\lambda^-=\left\{v\in C((-\infty, 0],  L^2(K)): \sup_{\tau\leq 0}e^{-\lambda\tau}\|v(\tau)\|_{L^2(K)}<\infty \right\}
\end{equation*}
endowed with norm
\begin{equation*}
\|v\|_{C_\lambda^-}=\sup_{\tau\leq 0}e^{-\lambda\tau}\|v(\tau)\|_{L^2(K)}\,,
\end{equation*}
where $-1/2<\lambda<0$\,.
In fact for any $v_0\in E_0$ define nonlinear operator $\mathcal{T}: C_\lambda^-\rightarrow C_\lambda^-$ by
\begin{equation*}
(v^*_0, v^*_s)\mapsto \mathcal{T}(v^*_0, v^*_s; v_0)=\text{right-hand side of~(\ref{e:v*c})--(\ref{e:v*s})}\,.
\end{equation*}
Then for $R>0$ is small enough, a direct calculation yields the mapping~$\mathcal{T}$ is contraction; that is, $\mathcal{T}$~has a unique fixed point $(\bar{v}_0^*, \bar{v}_s^*)=(v_0, \bar{v}_s^*)\in C_\lambda^-$ which is the unique solution to~(\ref{e:v*c})--(\ref{e:v*s}).

Then
\begin{equation}\label{e:LRIM}
 \mathcal{M}_{\text{loc}}(\omega)
 =\mathcal{M}^R_{\text{cut}}(\omega)\cap B_R(0)
 \end{equation}
defines a \textsc{lrim} for~$\phi(t,\omega)$. Furthermore, by the similar discussion for stochastic Burgers' equation on bounded domain~\cite{BlomkerWang10} , the random invariant manifold~$\mathcal{M}^R_{\text{cut}}(\omega)$ is almost surely complete. That is, we have the following result.
\begin{theorem}\label{thm:cutRIM}
Assume $R>0$ is small enough. Then for any solution $v^R(\tau,\xi)=(v^R_0(\tau, \xi), v^R_s(\tau,\xi))$ of the cutoff system~(\ref{e:v-c-cut})--(\ref{e:v-s-cut})\, there is
one orbit~$V^R(\tau ,\xi)$ on~$\mathcal{M}^R_{\text{cut}}$ with $V^R_0=v^R_0$ such that
\begin{eqnarray*}
&&\|(v^R_0(\tau), v^R_s(\tau))-(V^R_0(\tau), V^R_s(\tau))\|_{L^2(K)}\\&\leq & \|v^R(0)-V^R(0)\|e^{-\lambda^*\tau}
\end{eqnarray*}
where
\begin{equation*}
\lambda^*=\rat12-\frac34L_R\left(1+\frac{1}{\delta}\right)
\end{equation*}
for some $\delta>0$\,.
\end{theorem}

The above result yields that
\begin{theorem}\label{thm:LRIM}
For any $u^{0}\in B_R(0)$,
\begin{equation*}
{\rm dist}(\phi(\tau,\omega)u^{0}, \mathcal{M}_{{\rm loc}}(\theta_\tau\omega))\leq 2Re^{-\lambda^* \tau}
\end{equation*}
for all $\tau<\tau_0(\omega)=\inf\{\tau>0: \phi(\tau,\omega)u_0\not\in B_R(0)\}$\,.
\end{theorem}

Theorem~\ref{thm:cutRIM} and \ref{thm:LRIM} describe the local attractive property of the zero solution of~(\ref{e:RBurgers}). Especially for the solution~$(0, v_s(\tau,\xi))$ to~(\ref{e:RBurgers}), by $h(0,\omega)=0$\,, the attractive orbit on invariant manifold~$\mathcal{M}_{\text{loc}}(\omega)$ is the zero solution. That is the solution $u(\tau,\xi)=(u^{0}, u_s(\tau, \xi))$ to~(\ref{e:tau-xi-Burgers}) is attracted by a stationary solution $\bar{u}(\theta_\tau\omega, \xi)=(u^{0}, \bar{u}_s(\theta_\tau\omega,\xi))$ for $\tau<\tau_0(\omega)$ provide $u(0,\xi)-\bar{u}(0,\xi)$ lies in a small ball $B_R(0)\subset L^2(K)$. By the construction of the local random invariant manifold, $\bar{u}$~is the unique stationary solution that attracts~$u$. We then still have the result of Theorem~\ref{thm:u-ubar} in a small ball of $L^2(K)\cap L^\infty(\R)$. However,  the local random invariant manifold method does not restrict solution~$u$ to be  small in the space $L^2(K)\cap L^\infty(\R)$. We draw this conclusion in the following corollary.
\begin{coro}\label{cor:u-ubar}
For any solution $u(\tau,\xi)=(u^{0}, u_s(\tau,\xi))$ to equation~(\ref{e:tau-xi-Burgers}) with $u^{0}\in L^2(K)\cap L^\infty(\R)$, let $\bar{u}(\tau,\xi)=\bar{u}(\theta_\tau\omega, \xi)=(u_0, \bar{u}_s(\theta_\tau\omega,\xi))$ be one stationary solution of equation~(\ref{e:tau-xi-Burgers}). Then if $u(0,\xi)-\bar{u}(0,\xi)$ lies in a small ball $B_R(0)\subset L^2(K)$ with $R$~small enough, almost surely
\begin{equation*}
\|u(\tau,\xi)-\bar{u}(\tau, \xi)\|_{L^2(K)}\leq \|u(0,\xi)-\bar{u}(0,\xi)\|_{L^2(K)}e^{-\lambda^*\tau}
\end{equation*}
for $\tau<\tau_0$\,.
\end{coro}
Now the following theorem applies to the stochastic Burgers' equation~(\ref{e:txBurgers}). 
\begin{theorem}
For any solution~$\fu (t,x)$ to stochastic Burgers' equation~(\ref{e:txBurgers}), let $\bar{\fu }(t,x)$~be a self-similar solution to equation~(\ref{e:txBurgers}). Then if $\fu (1,x)-\bar{\fu }(1,x)$ lies in a small ball $B_R(0)\subset L^2(K)$ with $R$~small enough, almost surely
\begin{equation*}
\sqrt{t}\|\fu (t,x)-\bar{\fu }(t,x)\|_{L^2(\R)}\rightarrow 0\,,\quad t\rightarrow\infty\,.
\end{equation*}
\end{theorem}

\begin{remark}
The local random invariant manifold we constructed here  is a global one for the random dynamical system~$\phi(t,\omega)$  by the next global asymptotic convergence.
\end{remark}

\subsection{Globally asymptotic convergence to self-similar solutions}
\label{ss:gacsss}

We show the asymptotic convergence in probability  of any solution of stochastic Burgers' equation. For this we first  consider the Markov semigroup defined by the solution of the equation~(\ref{e:tau-xi-Burgers}).  

Denote by $\mathcal{M}$ the space consisting  all probability measures on space $L^{2}(K)\cap L^{\infty}(\R)$ and endow $\mathcal{M}$ with the topology of weak convergence.   Define continuous Markov semigroup $\{P_{\tau}\}_{\tau\geq 0}$ on $\mathcal{M}$
\begin{equation*}
P_{\tau}\mu(A)=\mathbb{P}\{u(\tau,\cdot)\in A\}\,, \quad \mu\in\mathcal{M}
\end{equation*}
for any Borel measurable  set $A\subset L^{2}(K)\cap L^{\infty}(\R)$ and $u(\tau, \cdot)$ is the solution to equation~(\ref{e:tau-xi-Burgers}) with initial value $u^{0}$ distributes as $\mu$\,.  The space $\mathcal{M}$ is  too larger for our purpose.   For this we introduce the following subspace 
\begin{equation*}
\mathcal{M}_{2}=\left\{\mu\in\mathcal{M}: \int_{L^{2}(K)\cap L^{\infty}(\R)}\|u\|^{2}_{L^{2}(K)}\mu(du)<\infty \right\}\,.
\end{equation*}

By the discussion in section \ref{sec:selfsimilar}\,,  there is a $\bar{\mu}\in\mathcal{M}_{2}$ such that $P_{\tau}\bar{\mu}=\bar{\mu}$\,, which is called stationary measure of $P_{\tau}$\,.  We next show that for any $\mu\in\mathcal{M}_{2}$ there is a unique stationary measure $\bar{\mu}\in\mathcal{M}_{2}$ such that $P_{\tau}\mu$ converges weakly to $\bar{\mu}$ as $\tau\rightarrow \infty$\,, that is 
\begin{equation*}
\int_{L^{2}(K)\cap L^{\infty}(\R)}f(u)P_{\tau}\mu(du)\rightarrow \int_{L^{2}(K)\cap L^{\infty}(\R)}f(u)\bar{\mu}(du),\quad \tau\rightarrow \infty
\end{equation*}
for any bounded continuous function $f: L^{2}(K)\cap L^{\infty}(\R)\rightarrow \R$\,.  

Associate the solution to equation (\ref{e:tau-xi-Burgers}) we choose  $\mu\in\mathcal{M}_{2}$ which has the form
\begin{equation}\label{e:split}
\mu=\delta_{M}*\mu_{s}
\end{equation}
where $\delta_{M}$ is some Dirac measure on $E_{0}$ and $\mu_{s}$ is supported on $E_{s}$\,.  Then consider the limit of $P_{\tau}\mu$ as $\tau\rightarrow\infty$\,. First by the estimates in section \ref{sec:selfsimilar} we have a measure $\bar{\mu}$  and subsequence $\tau_{n}$ with $\tau_{n}\rightarrow\infty$\,, $n\rightarrow\infty$, such that 
\begin{equation}\label{e:convergence}
P_{\tau_{n}}\mu\rightarrow \bar{\mu}\,, \quad n\rightarrow\infty.
\end{equation} 
To show the uniqueness of $\bar{\mu}$ need a contraction property  of the system~(\ref{e:tau-xi-Burgers}).  This is from the deterministic result~\cite{Zuazua94,Kim01}.
\begin{lemma}\label{lem:contraction}
For any $u^{1}$\,, $u^{2}\in L^{2}(K)\cap L^{\infty}(\R)$ with 
\begin{equation*}
\int_{\R}u^{1}(\xi)\,d\xi=\int_{\R}u^{2}(\xi)\,d\xi\,.
\end{equation*}
Let $u^{1}(\tau,\xi)$ and $u^{2}(\tau,\xi)$ be the solutions to   equation (\ref{e:tau-xi-Burgers}) with initial value $u^{1}$ and $u^{2}$ respectively. Then the  function 
\begin{equation*}
\phi(\tau)=\int_{\R}|u^{1}(\tau,\xi)-u^{2}(\tau,\xi)|\,d\xi
\end{equation*}
is strictly decreasing almost surely.
\end{lemma}
\begin{proof}
Let $U(\tau,\xi)=u^{1}(\tau,\xi)-u^{2}(\tau,\xi)$, then we have the following linear equation
\begin{equation*}
U_{\tau}=U_{\xi\xi}-\tfrac12[(u^{1}+u^{2}-\xi)U]_{\xi}\,.
\end{equation*}
By the estimates in section \ref{sec:selfsimilar}\,,
$(u^{1}+u^{2}-\xi)_{\xi}$ is bounded by a random constant. Then for any a fixed $\omega\in\Omega_{0}$\,, the result is followed by the discussion for deterministic system~\cite{Zuazua94,Kim01}.
\end{proof}

Now for any $\mu\in\mathcal{M}$ with form (\ref{e:split}),   let  $\bar{\mu}$ be  a stationary measure  and~(\ref{e:convergence}) holds. Suppose $\bar{\mu}'$ is another stationary measure of $P_{\tau}$ such that for some $\tau_{n}'\rightarrow \infty$\,, $n\rightarrow\infty$\,, 
\begin{equation}
P_{\tau_{n}'}\mu\rightarrow \bar{\mu}'\,,\quad n\rightarrow\infty\,.
\end{equation}
Denote by $\bar{u}(\tau,\xi)$ and $\bar{u}'(\tau,\xi)$ the solutions of equation (\ref{e:tau-xi-Burgers}) with initial value $\bar{u}^{1}(\xi)$ and $\bar{u}^{2}(\xi)$ distributes as $\bar{\mu}$ and $\bar{\mu}'$ respectively. Then
\begin{equation*}
\int_{\R}\bar{u}^{1}(\xi)\,d\xi=\int_{\R}\bar{u}^{2}(\xi)\,d\xi\,.
\end{equation*}
By Lemma \ref{lem:contraction}, the function 
\begin{equation*}
\int_{\R}|\bar{u}(\tau,\xi)-\bar{u}'(\tau,\xi)|\,d\xi
\end{equation*}
is almost surely strictly decreasing in $\tau$ which contradicts stationary of $\bar{u}$ and~$\bar{u}'$. Then we deduce the following result. 
\begin{theorem}\label{thm:global}
For any $u^{0}\in L^{2}(K)\cap L^{\infty}(\R)$, the solution $u(\tau,\xi)$ to equation~(\ref{e:tau-xi-Burgers}) with initial value $u^{0}$ converges in distribution\,, as $\tau\rightarrow\infty$\,,  to $\bar{u}$ in space $L^{2}(K)$ which is the unique stationary solution to equation (\ref{e:tau-xi-Burgers}) with 
\begin{equation*}
\int_{\R}\bar{u}(\tau,\xi)\,d\xi=\int_{\R}u^{0}(\xi)\,d\xi\,.
\end{equation*}
\end{theorem}

Next we show the above convergence in distribution is in fact a convergence in probability. For this we need the following fact stated  by Gy\"ongy and Krylov~\cite{GK96}.  Let $X$ be a Polish space with Borel sigma algebra. A sequence $\{X_{n}\}$ of $X$-valued random variables converges in probability if and only if for every pair of subsequences $\{X_{m}\}$ and $\{X_{l}\}$\,, there exists $X\times X$-valued subsequence $Z_{k}=(X_{m(k)}, X_{l(k)})$ converging  in distribution to a random variable $Z$ supported on $\{(x, y)\in X\times X: x=y \}$\,.  

Now for solution $u(\tau,\xi)$ to equation~(\ref{e:tau-xi-Burgers}), consider any subsequences $\{u(\tau_{m})\}$ and $\{u(\tau_{l})\}$. By the Theorem \ref{thm:global}, both subsequences converges in distribution to a random variable supported on 
\begin{equation*}
\{(u, v)\in (L^{2}(K)\cap L^{\infty}(\R))\times (L^{2}(K)\cap L^{\infty} (\R)): u=v\}.
\end{equation*}
Then we draw the following result
\begin{coro}\label{cor:global}
For any solution $u(\tau,\xi)$ to equation~(\ref{e:tau-xi-Burgers}) with initial value $u^{0}\in L^{2}(K)\cap L^{\infty}(\R)$, there is a unique stationary solution  $\bar{u}(\tau,\xi)$ such that  
 \begin{equation*}
\|u(\tau,\xi)-\bar{u}(\tau,\xi)\|_{L^{2}(K)}\rightarrow 0\,,\quad \text{in probability as}\ \tau\rightarrow\infty\,.
\end{equation*}
\end{coro}

\begin{remark}
By the  random invariant manifold discussion in section \ref{sec:rds}, the convergence rate approximates $1/2$ after a long time. Furthermore, by the above global asymptotic convergence result, any stationary solution is uniquely determined by its part in~$E_{0}$, that is for any $u_{0}\in E_{0}$\,, there is $u_{s}=h(u_{0},\omega)$ such that  $(u_{0}, h(u_{0},\omega))$ is the unique stationary solution with $u_{0}$ part in $E_{0}$\,. Then by the local random invariant manifold discussion in section \ref{sec:rds},  equation~(\ref{e:tau-xi-Burgers}) has a global random invariant manifold which can be represented by $\{(u_{0}, h(u_{0},\omega)): u_{0}\in E_{0}\}$\,.
\end{remark}

Then the following theorem applies to the stochastic Burgers' equation~(\ref{e:txBurgers}). 
\begin{theorem}
For any solution $\mathfrak{u}(t,x)$ to stochastic Burgers' equation~(\ref{e:txBurgers}), there is a unique self-similar solution $\bar{\mathfrak{u}}(t,x)$ to equation~(\ref{e:txBurgers}) such that 
\begin{equation*}
\sqrt{t}\|\mathfrak{u}(t,x)-\bar{\mathfrak{u}}(t,x)\|_{L^{2}(\R)}\rightarrow 0\,,\quad t\rightarrow\infty\,,
\end{equation*}
in probability. 
\end{theorem}

\section{Self-similarity emerges in examples of practical interest}
\label{s:sseepi}

The proof of the previous section~\ref{sec:sbe} shows one case where we can prove the emergence of stochastic self-similarity. 
However, the stochastic slow manifold framework of section~\ref{sec:ssme} strongly indicates that stochastic self-similarity emerges in a much wider class of stochastic systems. We proceed in the remaining sections to explore, albeit less rigorously, several example \spde{}s of the form~\eqref{eq:phit} and of significant physical interest.

We investigate particular stochastic systems and the self-similarity in their stochastic slow manifolds by writing the dynamics in the Hermite basis:  from the eigenbasis~\eqref{eq:eigen} we express the transformed field as the spectral expansion
\begin{equation}
u(\tau,\xi)=\sum_{k=0}^\infty u_k(\tau)e_k(\xi)=\sum_{k=0}^\infty u_k(\tau)c_kH_k(\xi/\sqrt2)G(\xi),
\label{eq:vexp}
\end{equation}
for normalisation coefficients~$c_k$ and mode amplitudes~$u_k(\tau)$.  We also write the $Q$-Wiener noise in the cylindrical expansion~\eqref{eq:dotw}. The very first difference with the rigorous analysis of section~\ref{sec:sbe} is that in this section we allow the noise to have a mean component in space; that is, in this section the noise coefficient $b_0\neq 0$\,.  The consequent direct forcing of the fundamental Gaussian mode makes an immediate qualitative difference to the long term evolution that applied scientists and engineers will appreciate.

\subsection{Stochastically forced diffusion}\label{subsec:diffusion}
Obtain the simplest stochastic self similarity for additive noise and when there is no nonlinearity: in this subsection we set $f=0$ and $g=1/t$ in \spde~\eqref{eq:prob}.  Substitute the Hermite expansions~\eqref{eq:vexp} and~\eqref{eq:dotw} into the transformed \spde~\eqref{eq:phit} and equate coefficients of the basis functions~$e_k(\xi)$ to find the decoupled \sde\ system
\begin{equation}
\dot u_k=-\rat k2u_k+b_k\dot w_k\,,\quad k=0,1,2,\ldots\,.
\label{eq:oulin}
\end{equation}
For mode numbers $k\geq1$\,, the \sde~\eqref{eq:oulin} describes an Ornstein--Uhlenbeck process and so the amplitudes of the fast modes are $u_k=u_k(0)e^{-k\tau/2}+b_k\Z{-k\tau/2}\dot w_k$ where we define the stochastic convolution 
\begin{equation*}
\Z{-\beta\tau}\dot w_k=\int_{0}^\tau e^{-\beta(\tau-\sigma)}dw_k(\sigma)
\end{equation*}
which is the Ornstein--Uhlenbeck process satisfying the \sde\ $dz=-\beta z\,d\tau+dw_k(\tau)$.  Consequently, exponentially quickly $u_k\to b_k\Z{-k\tau/2}\dot w_k$ as log-time $\tau\to\infty$\,, and this exponential approach is~\Ord{e^{-\tau/2}} due to the decay of the leading `fast' component of the initial compact release.  

But the slow mode satisfies $\dot u_0=b_0\dot w_0$ with solution $u_0=a(\tau)=b_0w_0(\tau)$.
Hence, in this case the emergent stochastic slow manifold of the \spde~\eqref{eq:phit} is
\begin{equation*}
u=\frac{a}{2\sqrt\pi}e^{-\xi^2/4}
+\sum_{k=1}^\infty e_k(\xi)\Z{-k\tau/2}\dot w_k
\end{equation*}
where the amplitude~$a=b_0w_0(\tau)=b_0w_0(\log t)$.
That is, under the direct forcing of the Gaussian structure, the stretched field~$u$ undergoes a random walk in the amplitude of the Gaussian, while exhibiting zero-mean fluctuations due to the past history of the other noise components.
Since the Wiener process $w_0(\tau)=\Ord{\tau^{1/2}}$, this predicts the original stochastic self similarity field will be $\fu=\Ord{t^{-1/2}(\log t)^{1/2}}$ as $t\to\infty$\,.

For comparison, recall that section~\ref{sec:sbe} proves the emergence of a stochastically stationary distribution in the case when there is no mean noise component.

The self-similar random walk emerges from transients of relative magnitude $\Ord{e^{-\tau/2}}=\Ord{t^{-1/2}}$; that is, of absolute magnitude in field~$\fu$ of~\Ord{t^{-1}}.  Section~\ref{sec:mstort} returns to this case to argue that stochastically moving the space origin and stretching time empowers us to eliminate the leading two stable modes $u_1$~and~$u_2$.  The new view of Section~\ref{sec:mstort} is a stochastic self-similarity that emerges somewhat quicker, with relative transients $\Ord{e^{-3\tau/2}}=\Ord{t^{-3/2}}$; that is, of absolute magnitude in field~$\fu$ of~\Ord{t^{-2}}.

\subsection{Cubic reaction enhances decay}

Here consider the case of cubic reaction, $f=-\fu^3$, and additive noise, $g=1/t$\,.  The stochastic slow manifold shows that not only does the cubic reaction aid the decay, but also noise-noise interactions increase the exponent in the self-similar decay.

\subsubsection{Change to Hermite basis}

For simplicity, initially just project the dynamics of the \spde~\eqref{eq:phit} onto the first three Hermite modes; section~\ref{sec:tssm3} uses computer algebra to implement more modes.  Recall the Hermite functions are $H_k(\zeta)=(-1)^ke^{\zeta^2/2}\Dn \zeta k{e^{-\zeta^2/2}}$: the first few are $H_0(\zeta)=1$\,, $H_1(\zeta)=\zeta$ and $H_2(\zeta)=\zeta^2-1$\,.  As given by  equation~\eqref{eq:vexp}, we expand the solution field in the corresponding basis.
The complication is the cubic reaction term
\begin{equation*}
u^3=\sum_{l,m,n}u_lu_mu_n c_lc_mc_n H_l(\rxi)H_m(\rxi)H_n(\rxi)\frac{e^{-3\xi^2/4}}{8\pi^{3/2}}\,.
\end{equation*}
We want to expand $u^3=\sum_k d_ke_k(\xi)$, so take the weighted inner product of this cubic reaction with $e_k=c_kH_k(\rxi)G(\xi)$ to determine the coefficient
\begin{align*}
d_k&{}=\sum_{l,m,n}u_lu_mu_n c_kc_lc_mc_n\int_{-\infty}^{\infty}  H_k(\rxi)H_l(\rxi)H_m(\rxi)H_n(\rxi)\frac{e^{-3\xi^2/4}}{16\pi^{2}}\,d\xi
\\&{}
=\sum_{l,m,n}u_lu_mu_n \frac{c_kc_lc_mc_n}{16\pi^{2}}\sqrt{\frac23}\int_{-\infty}^{\infty}  H_k(\rat\zeta{\sqrt3})H_l(\rat\zeta{\sqrt3})H_m(\rat\zeta{\sqrt3})H_n(\rat\zeta{\sqrt3}) {e^{-\zeta^2/2}}\,d\zeta
\end{align*}
upon substituting $\xi=\sqrt{2/3}\,\zeta$\,.  Expressing the Hermite functions as polynomials in~$\zeta$, expanding the quadruple products into a high order polynomial in~$\zeta$, and then recasting the polynomial in a sum of Hermite functions, the coefficient of the~$H_0(\zeta)$ mode determines the above integral.  Omitting details, upon truncating the above sums over modes to just the first three modes, $k,l,m,n=0,1,2$\,, we find coefficients in the cubic~$u^3$ are
\begin{align*}
&d_0=\frac1{\sqrt{3\pi}}\big(
\rat12u_0^3 +\rat12u_0u_1^2
-\rat{\sqrt2}9u_2^3 +\rat12u_0u_2^2 -\rat1{\sqrt2}u_0^2u_2
\big),
\\&
d_1=\frac1{\sqrt{3\pi}}\big(
\rat12u_0^2u_1 +\rat1{6}u_1^3
+\rat16u_1u_2^2
\big),
\\&
d_2=\frac1{\sqrt{3\pi}}\big(
-\rat{\sqrt2}{6}u_0^3 +\rat12u_0^2u_2 -\rat{\sqrt2}6u_0u_2^2
+\rat5{54}u_2^3 +\rat1{6}u_1^2u_2
\big).
\end{align*}
Consequently, the stochastic system, when projected onto the first three modes, is approximated by the set of \sde{}s
\begin{align}
&\dot u_0=\phantom{-\rat12u_0}+b_0\dot w_0
+\frac1{\sqrt{3\pi}}\big(
-\rat12u_0^3-\rat12u_0u_1^2
+\rat{\sqrt2}9u_2^3-\rat12u_0u_2^2+\rat1{\sqrt2}u_0^2u_2
\big) ,
\nonumber\\&\dot u_1=-\rat12u_1 +b_1\dot w_1
+\frac1{\sqrt{3\pi}}\big(
-\rat12u_0^2u_1-\rat1{6}u_1^3
-\rat16u_1u_2^2
\big) ,
\nonumber\\&\dot u_2=-u_2 +b_2\dot w_2
+\frac1{\sqrt{3\pi}}\big(
\rat{\sqrt2}{6}u_0^3-\rat12u_0^2u_2+\rat{\sqrt2}6u_0u_2^2
-\rat5{54}u_2^3-\rat1{6}u_1^2u_2
\big) .
\label{eq:3sde}
\end{align}
Computer algebra checks this derivation, and also computes the corresponding systems for higher order projections. 

\subsubsection{A normal form separates fast and slow modes}
\label{sec:nfsfsm}

As a preliminary to the more complete construction of the stochastic slow manifold in section~\ref{sec:tssm3}, this subsection shows that a stochastic coordinate transform~\cite[Ch.~8]{Arnold03} separates the stochastic slow and fast modes in the projected system~\eqref{eq:3sde}.

In the projected system~\eqref{eq:3sde} the dynamical variables~$u_k(\tau)$ are linearly diagonalised with constant coefficients, and the nonlinearities are of multinomial form.  In such a case, constructing the necessary stochastic coordinate transform is routine~\cite{Roberts06k}.  Indeed, a web service~\cite{Roberts07d} analyses the system~\eqref{eq:3sde} to construct a near identity, stochastic, coordinate transformation from variables~$u_k(\tau)$ to new variables~$U_k(\tau)$. The constructed transformation, beginning
\begin{align}
&u_0\approx U_0
+\frac1{\sqrt{3\pi}}\big(-\rat1{\sqrt2}U_0^2U_2+\rat1{4}U_0U_2^2
+\rat1{2}U_0U_1^2-\rat{\sqrt2}{54}U_2^3\big),
\nonumber\\&u_1\approx U_1
+\frac1{\sqrt{3\pi}}\big(\rat1{6}U_1^3+\rat1{12}U_1U_2^2\big)
+b_1 \Z{-\tau/2} \dot w_1\,,
\nonumber\\&u_2\approx U_2
+\frac1{\sqrt{3\pi}}\big(\rat{\sqrt2}{6}U_0^3
-\rat{\sqrt2}{6}U_0U_2^2
+\rat5{108}U_2^3+\rat1{6}U_1^2U_2\big)
+b_2 \Z{-\tau} \dot w_2\,,
\label{eq:3sxf}
\end{align}
transforms the system of \sde{}s~\eqref{eq:3sde} to the equivalent system
\begin{align}
\sqrt{3\pi}\dot U_0\approx{}&
-\rat12U_0^3 +\big(\sqrt{3\pi} b_0\dot w_0+\rat1{\sqrt2}U_0^2 b_2\dot w_2\big)
\nonumber\\&{}+U_0\left[ - b_1^2\dot w_1\Z{-\tau/2}\dot w_1
-\rat12 b_2^2\dot w_2\Z{-\tau}\dot w_2
+ \sqrt2b_0b_2\dot w_0\Z{-\tau}\dot w_2
\right],
\nonumber\\
\sqrt{3\pi}\dot U_1\approx{}&
-\rat12\sqrt{3\pi} U_1 -\rat12U_0^2U_1
\nonumber\\&{}+U_1\left[ - b_1^2\dot w_1\Z{-\tau/2}\dot w_1
-\rat16 b_2^2\dot w_2\Z{-\tau}\dot w_2
\right],
\nonumber\\
\sqrt{3\pi}\dot U_2\approx{}&
-\sqrt{3\pi} U_2-\rat12U_0^2U_2 
+\big(-\rat1{6}U_1^2 b_2\dot w_2
+\rat{\sqrt2}{3}U_0U_1 b_2\dot w_2\big)
\nonumber\\&{}+U_2\left[ -\rat13 b_1^2\dot w_1\Z{-\tau/2}\dot w_1
-\rat5{18} b_2^2\dot w_2\Z{-\tau}\dot w_2
+\rat13 b_0b_2\dot w_0\Z{-\tau}\dot w_2
\right].
\label{eq:3sdes}
\end{align}
Alternatively, one can straightforwardly confirm the order of accuracy of \eqref{eq:3sxf}~and~\eqref{eq:3sdes} by simply substituting them into the projected system of \sde{}s~\eqref{eq:3sde}.

The long term dynamics are readily apparent from the transformed system~\eqref{eq:3sdes}.  We immediately deduce the existence of a slow manifold, its emergence, and its evolution.
Observe that $U_1=U_2=0$ is invariant in the transformed system~\eqref{eq:3sdes}; a stochastic coordinate transform such as~\eqref{eq:3sxf} may always be found to create such invariance to any specified order~\cite{Arnold03b, Imkeller02, Roberts06k}. Due to the exponential decay of $U_1$~and~$U_2$ in the deterministic parts of~\eqref{eq:3sdes}, provided the magnitudes~$b_k$ of the stochastic effects are not too large, the invariant manifold $U_1=U_2=0$ will be almost surely exponentially quickly attractive.  The emergent stochastic slow manifold is thus $U_1=U_2=0$\,.  Evolution on the stochastic slow manifold is given by the first line of~\eqref{eq:3sdes}, for all time. Substituting $U_1=U_2=0$ in the transform~\eqref{eq:3sxf} then gives the shape of the stochastic slow manifold in the $u_k$-variables.

The transformed \sde{}s~\eqref{eq:3sdes} indicate that the cubic nonlinearity enhances the rate of attraction to the stochastic slow manifold: the deterministic part of the two fast modes are $\dot U_k\approx-\big[\rat12 k+\frac12U_0^2/(\sqrt3\pi)\big]U_k$\,.  But the most important aspect is that, by continuity, there exists a finite domain near the $u=0$ equilibrium such that only very rare stochastic events could overcome the exponential attraction to the stochastic slow manifold.  Thus we expect the stochastic slow manifold to almost always emerge from some finite domain of initial conditions.

Consequently, the following `shadow' modelling applies.  The \sde\ system~\eqref{eq:3sdes} and the stochastic transform~\eqref{eq:3sxf} together describe a stochastic process in the $u_k(\tau)$~coefficients of the Hermite basis functions~$e_k(\xi)$, and thus describes a stochastic process in the field~$u(\tau,\xi)$.  By the asymptotic construction, this stochastic process is in an asymptotic sense `close to' or `shadows' the original \spde~\eqref{eq:phit}, especially when carried out to better resolution as we do in the next section~\ref{sec:tssm3}.  Now, in physical applications the original coefficients, functional forms and noise spectrum in the \spde~\eqref{eq:phit} are never known exactly.  Thus in physical applications, predictions deduced from systematic shadowing stochastic processes such as \eqref{eq:3sdes}~and~\eqref{eq:3sxf}, are as useful as predictions from the \spde~\eqref{eq:phit}.

On the emergent stochastic slow manifold, the system~\eqref{eq:3sde} evolves in the long term according to the first \sde\ of the system~\eqref{eq:3sdes}.  The most important part of the quadratic noise-noise interaction terms in~\eqref{eq:3sdes} is their effect upon the mean drift.  Analysis of such noise-noise interactions~\cite{Roberts05c} shows they generate drift and fluctuations: neglecting these fluctuations the slow mode of the \sde{}s~\eqref{eq:3sdes} becomes
\begin{equation}
\sqrt{3\pi}\dot a\approx
-(\rat12b_1^2+\rat14b_2^2)a-\rat12a^3 
+\sqrt{3\pi} b_0\dot w_0
+\rat1{\sqrt2}a^2 b_2\dot w_2\,,
\label{eq:ssm3}
\end{equation}
where we use $U_0=a$ to denote the amplitude of the stochastic self similar solution.
In this \sde: the direct forcing~$b_0\dot w_0$ promotes a random walk among the self-similar profiles, as for linear diffusion; the cubic reaction~$-\rat12a^3$ reflects the cubic reaction of the original physics; but the  $-(\rat12b_1^2+\rat14b_2^2)a$ term accounts for noise-noise fluctuations enhancing the exponent of the similarity decay rate.  Just solving $\sqrt3\pi\dot a\approx -(\rat12b_1^2+\rat14b_2^2)a$ gives amplitude $a\propto e^{-\alpha\tau}=t^{-\alpha}$ for exponent $\alpha=(\rat12b_1^2+\rat14b_2^2)/\sqrt{3\pi}$ and hence predicts the original field $\fu\propto t^{-1/2}a \propto t^{-1/2-\alpha}$.
Physically, because of the nonlinear shape of the cubic reaction, and in comparing fluctuations that enhance the local field with those that decrease the local field, the first generates reactions that are slightly larger than the reaction is decreased by the second.  Thus such a cubic reaction enhances the similarity decay rate through noise-noise interactions.  Our systematic resolution of the noise-noise interactions discerns this noise enhanced decay.

\subsubsection{The emergent slow manifold of stochastic self-similarity}
\label{sec:tssm3}

When, as for the \spde~\eqref{eq:phit}, there is a large number of noise excited fast modes, then the full normal form coordinate transform is impossible to construct. Fortunately, we just need to construct the stochastic slow manifold part of the coordinate transform, as is summarised in this section. Here we project the \spde~\eqref{eq:phit} onto the first nine modes and construct and interpret the resulting stochastic slow model.

Computer algebra obtains a system of \sde{}s for $\dot u_0,\ldots,\dot u_8$, analogous to the system~\eqref{eq:3sde}, but which are far too involved to record here.  Established methods~\cite{Roberts05c, Roberts06k}, available in computer algebra via the web~\cite{Roberts09c}, then analyses the system of \sde{}s to determine that the stochastic slow manifold is approximately
\begin{align*}
&&&u_0\approx a+\frac{a^2}{\sqrt{3\pi}}
\big[-\rat1{\sqrt2}b_2\Z{-\tau}\dot w_2
+\rat1{2\sqrt6}b_4\Z{-2\tau}\dot w_4
\\&&&\qquad\qquad\qquad{}
-\rat{\sqrt5}{27}b_6\Z{-3\tau}\dot w_6
+\rat{\sqrt{70}}{216}b_8\Z{-4\tau}\dot w_8\big],
\\&u_1\approx  b_1\Z{-\tau/2}\dot w_1 \,,
&&u_2\approx +\frac{a^3}{3\sqrt{6\pi}}+ b_2\Z{-\tau}\dot w_2 \,,
\\&u_3\approx  b_3\Z{-3\tau/2}\dot w_3 \,,
&&u_4\approx -\frac{a^3}{18\sqrt{2\pi}}+ b_4\Z{-2\tau}\dot w_4 \,,
\\&u_5\approx  b_5\Z{-5\tau/2}\dot w_5 \,,
&&u_6\approx +\frac{\sqrt5a^3}{81\sqrt{3\pi}}+ b_6\Z{-3\tau}\dot w_6 \,,
\\&u_7\approx  b_7\Z{-7\tau/2}\dot w_7 \,,
&&u_8\approx -\frac{\sqrt{70}a^3}{648\sqrt{3\pi}}+ b_8\Z{-4\tau}\dot w_8 \,.
\end{align*}
As shown explicitly for the low order system in section~\ref{sec:nfsfsm}, by continuity, and except for rare events, this stochastic slow manifold will almost always emerge from all initial conditions in its neighbourhood.

The methodology implemented in the web service~\cite{Roberts09c} also constructs the stochastic evolution on this stochastic slow manifold.  Effects quadratic in noise are very complicated---too complicated to record here---due to the need to resolve the multitude of noise-noise interactions that occur in the fast modes~\cite{Roberts05c}. Here we retain only their cumulative drift effects of the noise-noise interactions. Computer algebra then finds the evolution to be the following more complete version of the earlier~\eqref{eq:ssm3}:
\begin{align}
\sqrt{3\pi}\dot  a\approx{}&
-\rat12 a^3
+\sqrt{3\pi} b_0\dot w_0
\nonumber\\&{}
+ a^2\left[\rat1{\sqrt2}b_2\dot w_2
-\rat{1}{2\sqrt6}b_4\dot w_4
+\rat{\sqrt5}{27}b_6\dot w_6
-\rat{\sqrt{70}}{216}b_8\dot w_8\right]
\nonumber\\&{}-a\left[\rat12b_1^2
+\rat14b_2^2 +\rat7{54}b_3^2 +\rat{19}{216}b_4^2
+\rat{17}{270}b_5^2 +\rat{47}{972}b_6^2
+\rat{131}{3402}b_7^2 +\rat{41}{1296}b_8^2
\right].
\label{eq:ssm9}
\end{align}
The first line of the stochastic slow mode~\eqref{eq:ssm9} contains the direct effects of the cubic reaction and the stochastic forcing.  The second line is a multiplicative noise term that could be replaced by one independent noise term with volatility coefficient being $\sqrt{b_2^2/2 +b_4^2/24 +5b_6^2/729 +70b_8^2/46656}$. The last line of~\eqref{eq:ssm9} enhances the self-similarity decay rate through noise-noise interactions.

A Domb--Sykes plot~\cite[e.g.]{Mercer90} of the ratio of the coefficients in the last line suggests the corresponding infinite series converges provided assumption~\eqref{e:traceQ} holds.    However, extant stochastic slow manifold theory is limited to effectively finite dimensional dynamics such as the \spde~\eqref{eq:phit} projected onto the first nine Hermite modes that we analyse here.  Nonetheless, in principle we could construct the stochastic slow manifold model to some level of approximation for any finite truncation of the noise~\eqref{eq:dotw}.

\subsection{Modelling other stochastic systems}

Section~\ref{sec:ssme} established that a stochastic slow manifold approach illuminates a broad class of stochastic reaction-diffusion \pde{}s.  This approach is then supported by some rigorous theory in Section~\ref{sec:sbe} in a limited case, and illustrated in a relatively formal approach to an example in this section.  In the analysis here the major issue was purely the algebraic complexity.  Thus the main outcome of this section is to empower others to analyse any of the broad class of stochastic reaction-diffusion \pde{}s identified in section~\ref{sec:ssme} that may be of interest in specific applications.

\section{Vary the origin of space-time to improve modelling}
\label{sec:mstort}

Here reconsider linear diffusion with stochastic forcing 
\begin{equation}
\fu_t=\fu_{xx}+\fB(t,x),
\label{eq:diff}
\end{equation}
on an infinite spatial domain in one dimension and for some yet to be defined noise process~$\fB$.  Indeed the analysis of the next two sections almost always uses classic calculus and so also apply to deterministic, time dependent, forcing~$\fB$ as well as to stochastic forcing. Following a compact release of material it is natural to place the origin of the spatial coordinate system at about the location of the release.  Then the spread of material over time will be approximately symmetric about the spatial origin.  In deterministic diffusion-based systems Suslov and Robert~\cite{Suslov98a} showed how to choose the space origin for a given compact release: the optimal choice eliminates the slowest $t^{-1}$-transients in the approach to the self similarity solution.

However, when material is stochastically added\slash removed\slash moved over time, a marked asymmetry in the distribution about the origin generically develops dynamically: indeed, in stochastic dynamics such asymmetry is likely to be a random walk, as confirmed below, that grows stochastically like~$\sqrt t$.  Here we allow the reference point of the stochastic self-similarity to evolve in time to cater for such overall movement of the effective origin in space of the self-similar regime.  The proposed choice removes the longest lasting memory integrals in the stochastic self-similar solution.

Analogously it is expedient to change time.  In deterministic diffusion-based systems, Suslov and Robert~\cite{Suslov98a} showed how choosing the origin in time empowers one to eliminate the next slowest $t^{-3/2}$-transients in the approach to self similarity by better matching the variance. Correspondingly, for stochastic systems we seek to remove the corresponding memory integrals.  But, instead of expressing the adaptation as a change in the origin of time, here we allow the relationship between real time and effective self similarity time to vary dynamically.

Let's generalise the log-time transformation~\eqref{eq:gene}~\cite{Wayne94}.  Here scale the solution and space-time by
\begin{equation}
\tau=\log T\,,\quad
\xi=\frac{x-X(T)}{\sqrt T}\,,\quad
t=t(T),\quad
\fu=\frac1{\sqrt T}u(\tau,\xi).
\label{eq:genes}
\end{equation}
Then $X$~gives the effective centre in space of the spreading material at any time, and $T$~defines a pseudo-time that accounts for modifications to the effective width of the spreading material; we expect $T\approx t$ to some level of approximation. Under the coordinate transformation~\eqref{eq:genes}, partial derivatives become
\begin{equation*}
\D x{}=\frac1{\sqrt T}\D \xi{}
\quad\text{and}\quad
\D t{}=\frac{{T'}}{T}\left[ \D \tau{}
-\left(\rat{1}{2}\xi+{X'}{\sqrt T}\right)\D \xi{}\right].
\end{equation*}
We reserve overdots for the derivative~$d/d\tau$ so use $X'=dX/dT$ and $T'=dT/dt=1/t'$. Substituting into the \spde~\eqref{eq:prob} gives
\begin{align*}&
-\frac{{T'}}{2T^{3/2}}u+\frac{{T'}}{T^{3/2}}\left[ \D \tau u
-\left(\rat{1}{2}\xi+{X'}{\sqrt T}\right)\D \xi{u}\right]
=\frac1{T^{3/2}}\DD\xi u
+\fB(t,x)\,.
\end{align*}
Rearranging, using the definition~\eqref{eq:lphi} of operator~$\cL$, gives the \spde\ in similarity variables as
\begin{align*}
u_\tau={}&
\cL u +X'\sqrt Tu_\xi+(t'-1)u_{\xi\xi}
+t'T^{3/2}\fB(t,x).
\end{align*}
Now assume that the original noise process is such that $t'T^{3/2}\fB(t,x)=\dot W(\tau,\xi)$ for some cylindrical $Q$-Wiener process~$W(\tau,\xi)$.  Then the \spde\ in similarity variables becomes
\begin{equation}
u_\tau=\cL u +X'\sqrt Tu_\xi+(t'-1)u_{\xi\xi}
+\dot W.
\label{eq:xspde}
\end{equation}

\paragraph{Write in the Hermite basis}
Use the cylindrical expansion~\eqref{eq:dotw} for the noise~$\dot W$ and the spectral expansion~\eqref{eq:vexp} for the similarity field~$u(\tau,\xi)$.  The properties of Hermite polynomials imply that the basis function derivatives $e_{k\xi}=-e_{k+1}/\sqrt2$ and $e_{k\xi\xi}=e_{k+2}/2$\,.  Then equating coefficients of~$e_k$, the linearised version of the above \spde~\eqref{eq:xspde} becomes the component \sde{}s
\begin{equation*}
\dot u_k=-\rat 12ku_k-\rat1{\sqrt 2}X'\sqrt Tu_{k-1}+\rat12(t'-1)u_{k-2}+b_k\dot w_k\,.
\end{equation*}
In particular, the first three \sde{}s are
\begin{align}
&\dot u_0=b_0\dot w_0\,,\label{eq:v0w}
\\&\dot u_1=-\rat12u_1-\rat1{\sqrt 2}X'\sqrt Tu_{0}+b_1\dot w_1\,, \label{eq:v1w}
\\&\dot u_2=-u_2-\rat1{\sqrt 2}X'\sqrt Tu_{1}+\rat12(t'-1)u_{0}+b_2\dot w_2\,.
\label{eq:v2w}
\end{align}
As in section~\ref{s:sseepi}, solving the first \sde~\eqref{eq:v0w} gives the random walk of the amplitude $u_0=u_0(0)+b_0w_0(\tau)$.  However, interesting results become clearer in this section by making $b_0=0$\,, the case of conservative noise, so that $u_0=a$ is constant.  

\paragraph{Eliminate the leading transient}
With static~$X$ the solution of the second \sde~\eqref{eq:v1w} would involve fluctuations generated by the memory convolutions~$\Z{-\tau/2}\dot w_1$.  It is these fluctuations we remove by varying~$X$. Let's explore the case $b_0=0$ and $u_0=a$\,. Then setting $-\rat1{\sqrt 2}X'\sqrt Tu_{0}+b_1\dot w_1=0$ eliminates the forcing of~$u_1$ in the second \sde~\eqref{eq:v1w}.  That is, set the spatial `origin' $X=\sqrt 2b_1\int u_0^{-1}\sqrt T\,dw_1(\tau)$.  Because $\tau=\log T$\,, then $dw_1(\tau)=dw_1(T)/\sqrt T$\,.  Thus remove the forcing of~$u_1$ in~\eqref{eq:v1w} by choosing 
\begin{equation}
X=X(0)+\sqrt 2b_1\int_0^T \frac1{u_0}\,dw_1(T_1)
=X(0)+\frac{\sqrt2b_1}{a}w_1(T).
\end{equation}
Additionally choosing the initial condition~$X(0)$ appropriately~\cite{Suslov98a}, then causes the mode $u_1=0$ for all time.  With these choices for the dynamically varying spatial `origin'~$X(T)$ of the similarity transform we eliminate both the transient in~$u_1$ of relative magnitude~$1/\sqrt T$ and we eliminate all fluctuations in the~$u_1$ mode.  This recognises that the noise moves the effective centre of the material.

\paragraph{Eliminate the next transient} 
With a fixed~$t=T$ the third \sde~\eqref{eq:v2w} would involve fluctuations~$\Z{-\tau}\dot w_2$ that we now remove by varying the relation between time and pseudo-time.
Setting $-\rat1{\sqrt 2}X'\sqrt Tu_{1}+\rat12(t'-1)u_{0}+b_2\dot w_2=0$ then eliminates the forcing of~$u_2$ in the third \sde~\eqref{eq:v2w}.  Assume we chose the `space origin'~$X(T)$ so that $u_1=0$\,.  Then rearrange to $dt/dT=t'=1+2b_2\dot w_2/u_0$\,.  Consequently, using $dw_1(\tau)=dw_1(T)/\sqrt T$\,, remove the forcing of~$u_2$ in~\eqref{eq:v2w} by choosing real-time
\begin{equation}
t=t(0)+\int T\,d\tau+2b_2\int \frac{T}{u_0}\,dw_2(\tau)
=t(0)+T+\frac{2b_2}a\int_0^T\sqrt{T_1}\,dw_2(T_1).
\label{eq:sqrttdw}
\end{equation}
Figure~\ref{fig:sqrttdw} shows five realisations of an example of the relationship~\eqref{eq:sqrttdw} between real- and pseudo-time.
Additionally choosing the initial time~$t(0)$ of the coordinate transform appropriately~\cite{Suslov98a}, then causes $u_2=0$ for all time.  With these choices for the `origin'~$X(T)$ of the similarity transform, and the evolution of the pseudo-time~$T$ we eliminate both the transients of relative magnitude~$1/\sqrt T$ and~$1/T$, and all the fluctuations in the $u_1$~and~$u_2$ modes.

\begin{figure}
\centering
\begin{tabular}{c@{\ }c}
\rotatebox{90}{\hspace{15ex}real-time $t$} &
\includegraphics{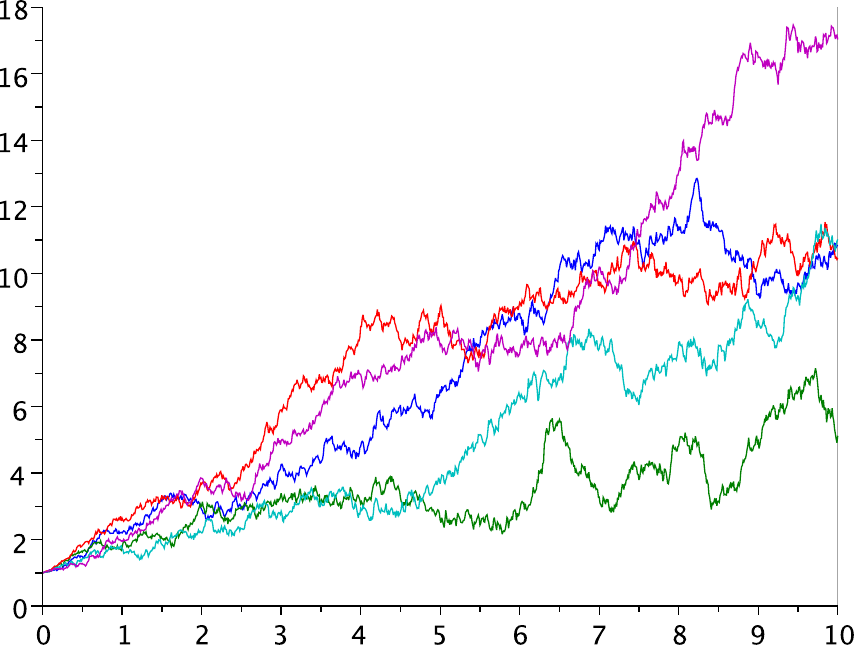}\\
& pseudo-time $T$
\end{tabular}
\caption{five realisations of the integral~\eqref{eq:sqrttdw} for the real-time versus pseudo-time relationship with parameters $t(0)=1$ and $2b_2/a=0.5$\,.}
\label{fig:sqrttdw}
\end{figure}

The pseudo-time is only very roughly linear in real-time (Figure~\ref{fig:sqrttdw}). Although real-time~$t$ does dominantly grow directly with pseudo-time~$T$, the stochastic integral provides significant fluctuations to the relationship.  The expectation $\EX \int_0^T\sqrt{T_1}\,dw_2(T_1)=0$ so that we deduce $\EX(t)=t(0)+T$.  However, It\^o's isometry shows
\begin{equation*}
\Var \int_0^T\sqrt{T_1}\,dw_2(T_1)=\int_0^TT_1\,dT_1=\rat12T^2.
\end{equation*}
That is, the size of the fluctuations in the pseudo-time about real time grows linearly in time.

This ability to optimally move the `origin' of space-time should also apply to stochastic self-similarity diffusion-like problems involving nonlinearity or multiplicative noise provided these effects are small enough perturbations.

\section{Stochastic advection and exchange in fluctuating time}
\label{s:saxft}

Turbulent mixing in fluids is vitally important and yet still incompletely understood.  One route to understand such turbulent processes is to explore mixing in a prescribed, but random, shear flow.  Majda~\cite{Majda93} begun exploring a model of mixing in a flow of a linear shear multiplied by a Gaussian white noise process. Majda, McLaughlin, Camassa et al.~\cite{McLaughlin96, Bronski2000, Camassa2008} continued exploring aspects of the stochastic aspects of the mixing.  They focussed on initial conditions which are statistically stationary in space.  In contrast, here we look at the problem of mixing in a stochastic shear flow from a compact release in space.  

As in section~\ref{sec:mstort} we find a fluctuating pseudo-time naturally arises in the emergent stochastic self-similarity.  The `mean' self-similarity displays classic diffusive growth to correspond to eddy diffusivity models of turbulence.  However, due to the stochastic shear flow, modelling fluid eddies, the concentration fields often partially reconstitute earlier times.  Such reconstitution is an anomalous diffusion which here we resolve via a fluctuating pseudo-time in the stochastic self-similarity.

\begin{figure}
\centering
\setlength{\unitlength}{1ex}
\begin{picture}(50,8)(0,0)
\thicklines
\put(0,7){$\fu_1(t,x)$}
\multiput(0,6)(5,0){10}{\vector(1,0){5}}
\put(0,0){$\fu_2(t,x)$}
\multiput(5,2)(5,0){10}{\vector(-1,0){5}}
\put(7,3.5){$\Uparrow\Downarrow$ exchange}
\put(25,3.5){$\pm \dot w(t)$ stochastic advection}
\end{picture}
\caption{schematic diagram of a stochastic shear flow in two `pipes' carrying some material with concentrations~$u_j(t,x)$ which exchanges between the `pipes' and is advected by a white noise velocity~$\pm\dot w(t)$.}
\label{fig:toyturb}
\end{figure}
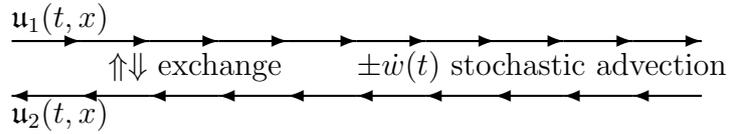

But to make progress in this first treatment we simplify the shear layer even further than Majda, McLaughlin, Camassa et al.~\cite{McLaughlin96, Bronski2000, Camassa2008}.  As shown in Figure~\ref{fig:toyturb}, we simplify by having just two layers, or `pipes' in which the advecting velocity field is formally Gaussian white noise~$\pm \dot w(t)$, equal and opposite in each `pipe'.  Notionally, these pipes correspond to near neighbouring streamlines in a turbulent fluid flow.  Some material of concentration~$u_j(t,x)$ in each `pipe' is exchanged between the `pipes' at a non-dimensional rate~$1/2$ ---analogous to diffusion across streamlines in turbulent flow. The governing stochastic \pde{}s for this system, in the Stratonovich interpretation, are thus the stochastic advection-exchange equations
\begin{equation} 
\D t{\fu_1}=\rat12(\fu_2-\fu_1)-\dot w(t) \D x{\fu_1} \,,\quad
\D t{\fu_2}=\rat12(\fu_1-\fu_2)+\dot w(t) \D x{\fu_2}\,.
\label{eq:u12}
\end{equation}
Now let's explore the stochastic self-similarity modelling of these dynamics following a compact release of material around $x=0$ at some initial time.

The self similarity appears with independent variables of  stretched space~$\xi$ and log-pseudo-time~$\tau$ defined by
\begin{equation}
\xi=x/\sqrt T\,,\quad
\tau=\log T\,,\quad
\dot T=\eta \dot w \,,\quad
\dot \eta =-\eta +\dot w
\label{eq:ssiv}
\end{equation}
(using slightly different notation in that in this section overdots denote time derivatives~$d/dt$).  Because of the symmetry in the stochastic advection, there is no need to seek a moving origin in space~$x$ that we addressed in the previous section~\ref{sec:mstort}.  The pseudo-time~$T$ is defined in terms of the auxiliary Ornstein--Uhlenbeck process~$\eta$, that is then combined with the forcing again to drive~$T$ (interpret in the Stratonovich sense). Curiously, the $T\eta$-system is a well-known irreproducible `kernel' in stochastic slow manifold analysis: Chao and Roberts~\cite[\S4]{Chao95} identified that the $T\eta$-system is its own stochastic slow manifold model; they analysed the corresponding Fokker--Planck equation to argue that on long time scales $dT=\rat12\,dt+\rat1{\sqrt2}dw_1$ for effectively independent noise~$dw_1$. This long time scale \sde\ accounts for the mean growth of pseudo-time~$T$ with time~$t$, and the fluctuations thereon, as shown by Figure~\ref{fig:sde1b}.  There is a good physical reason for the `reversals' in~$T(t)$ displayed by Figure~\ref{fig:sde1b}: differential advection by a `turbulent' fluctuation of a lump of material at some location spreads material; then a reversal of the fluctuation reconstitutes much of the earlier distribution.  Reversals in~$T(t)$ represent such reconstitution.

In addition to transforming independent variables, we transform dependent fields to the mean and difference variables 
\begin{align*}&
\text{(mean)}&&
u(\tau,\xi)/\sqrt T=\rat12(\fu_1+\fu_2),
\\&
\text{(difference)}&&
v(\tau,\xi)\eta /T=\rat12(\fu_1-\fu_2).
\end{align*}
The mean is scaled to decay like~$1/\sqrt T$ as in classic dispersion, but the difference field is scaled to decay faster, like~$1/T$, and moderated by the Ornstein--Uhlenbeck process~$\eta$.  In modelling turbulent mixing, the macroscopic dynamics of the mean field is of prime interest.

\begin{figure}
\centering
\begin{tabular}{c@{\ }c}
\rotatebox{90}{\hspace{10ex}pseudo-time $T$} &
\includegraphics{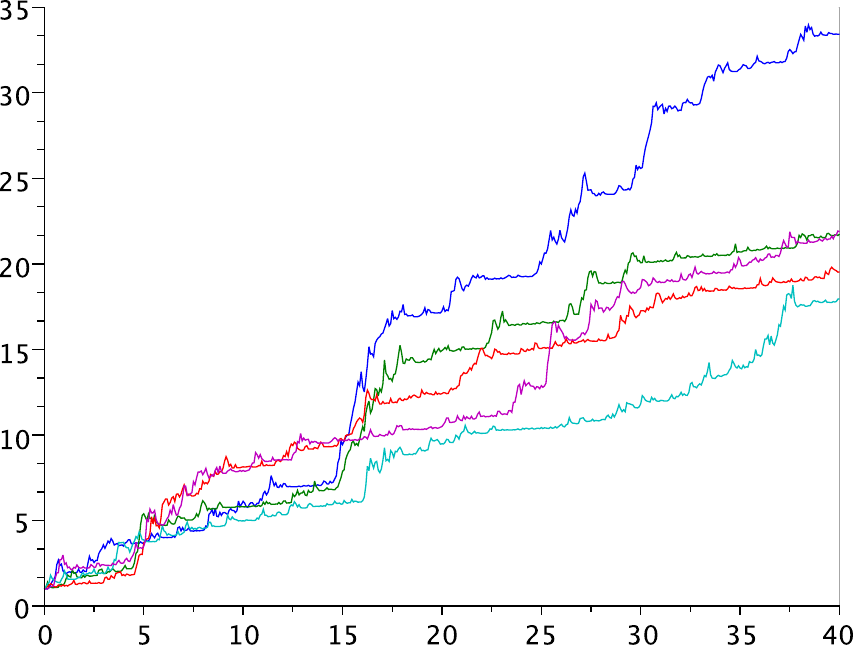}\\
& real-time $t$
\end{tabular}
\caption{five realisations of the pseudo-time~$T$ as a function of real-time~$t$.  The `reversals' in the relationship correspond to reversals in advection partially reconstituting the distribution of material at an earlier time.}
\label{fig:sde1b}
\end{figure}

Stochastic self-similarity emerges from the stochastic `turbulent' mixing.
Straightforward change of variable algebra derives that derivatives
\begin{equation*}
\D t{}=\frac{\eta\dot w}T\left(\D\tau{}-\rat12\xi\D\xi{}\right)
\quad\text{and}\quad
\D x{}=\frac1{\sqrt T}\D\xi{}.
\end{equation*}
Substituting these transformations into the advection-exchange \pde~\eqref{eq:u12} and rearranging, the transformed system for mean and difference variables in log-pseudo-time is
\begin{align} &
\D\tau u=\rat12u+\rat12\xi\D\xi u-\D\xi v\,,
\label{eq:dut}
\\&
v=-\D\xi u-e^{-\tau}\eta ^2\left[ \D\tau v-\rat12\xi \D\xi v-v\right].
\label{eq:dvt}
\end{align}
The second \spde~\eqref{eq:dvt} indicates that exponentially quickly in log-pseudo-time~$\tau$, the difference field $v\to -\D\xi u$\,. Substitute this limit into the first \pde~\eqref{eq:dut} and it becomes
\begin{equation}
\D\tau u=\cL u=\rat12u+\rat12\xi\D\xi u+\DD\xi u\,.
\label{eq:clu}
\end{equation}
This well known linear \pde\ thus describes the emergent dynamics of the advection-exchange system~\eqref{eq:u12}.

The solutions of the \pde~\eqref{eq:clu} settle on a Gaussian${}\propto G(\xi)$. Writing the field~$u(\tau,\xi)$ in the spectral expansion~\eqref{eq:vexp}, and then equating coefficients in the \pde~\eqref{eq:clu} leads to the system $\dot u_k=-\rat12ku_k$\,, $k=0,1,2,\ldots$\,. Hence all modes tend to zero like~$e^{-\tau/2}$ or quicker, except the $k=0$ mode which is constant.  Thus $u= aG(\xi)+\Ord{e^{-\tau/2}}$ for some constant~$a$. The attraction of these equilibria in the transformed variables predicts the emergence of the classic self-similar spread in physical variables that the mean concentration $\rat12(\fu_1+\fu_2)= aT^{-1/2}G(x/\sqrt T)+\Ord{1/T}$ as $t\to \infty$\,, albeit expressed in fluctuating pseudo-time.

This generic emergence of the spreading Gaussian is an appealing parallel with turbulent eddy diffusivity models.  The fluctuations in the stochastic self-similarity emphasise the difficulty of just one part of traditional deterministic models of turbulent mixing.  In real turbulence the situation is vastly more complex in that we picture many stochastic eddies occurring together.  In such a situation the fluctuations that each may generate individually may well average to a smaller net effect.  Further research is needed.

\section{Conclusions}

We demonstrated that centre manifold theory provides a straightforward and rigorous way of deriving the functional form of similarity solutions of nonlinear stochastic diffusion, and proving the emergence of stochastic similarity from quite general compact initial conditions.  In particular, sections~\ref{sec:sbe} and~\ref{s:sseepi} derived explicit results for a stochastic Burgers' equation and a stochastic cubic reaction-diffusion equation.  Section~\ref{sec:mstort} then showed that we could vary in time the location of the origin in space, and vary the rate of time, to optimally describe the stochastic self-similarity.  The last section~\ref{s:saxft} then used an analogous stochastic similarity transform to illustrate the emergence of an anomalous eddy diffusion process in a toy turbulent mixing problem. These techniques appear promising for useful modelling a wide class of stochastic nonlinear diffusion-like problems.

The analysis of the last two sections \ref{sec:mstort}~and~\ref{s:saxft} involves purely classical calculus and so also applies to deterministic~$w(t)$.

\paragraph{Acknowledgements}
This research was supported by the Australian Research Council grants DP0774311 and DP0988738.

\ifcase1
\bibliography{ajr,bib}\bibliographystyle{plain}
\or

\fi

\end{document}